\documentclass[a4paper,12pt]{article}
\usepackage{amssymb,amsfonts,amsthm,mathrsfs}
\usepackage[namelimits,reqno]{amsmath}
\usepackage{graphicx}
\usepackage{float}
\usepackage[hidelinks=true]{hyperref}
\usepackage{textcomp}
\usepackage{verbatim}
\usepackage{enumerate}
\usepackage{hyperref}
\usepackage[T1]{fontenc}
\usepackage{authblk}
\usepackage{a4,graphicx}
\usepackage{tikz}
\usepackage{lipsum}
\usepackage{graphicx,subfigure}
\usepackage[makeroom]{cancel}
\usepackage{pifont}
\usepackage{color}
\allowdisplaybreaks[3]
\usepackage{vmargin}
\usepackage{fancyhdr}
\usepackage{epsfig}
\usepackage{a4,graphicx}
\usepackage{epsf,epic}
\usepackage{tabularx}
\usepackage{enumitem}
\makeatletter
\def\namedlabel#1#2{\begingroup
    #2%
    \def\@currentlabel{#2}%
    \phantomsection\label{#1}\endgroup
}
\makeatother
\usepackage[T1]{fontenc}
\usepackage[Conny]{fncychap}

\usepackage{inputenc}

\usepackage{geometry}

\theoremstyle{plain}
\newtheorem{theorem}{Theorem}[section]
\newtheorem{lemma}{Lemma}[section]

\newtheorem{corollary}{Corollary}[section]
\theoremstyle{definition}
\newtheorem{example}{Example}[section]

\newtheorem{remark}{Remark}[section]
\newcommand{\R}{\mathbb R}

\newcommand{\calL}{\mathcal{L}}

\newcommand{\calF}{\mathcal{F}}
\def\a{\alpha}
\def\b{\beta}
\def\g{\gamma}

\def\t{\theta}

\def\O{\Omega}

\def\d{\delta}

\def\r{\rho}
\def\s{\sigma}
\def\l{\lambda}
\def\vp{\varphi}

\def\ve{\varepsilon}

\setmarginsrb{2cm}{2cm}{2cm}{2cm}{0.5cm}{0.5cm}{1.5cm}{1.5cm}
\if@twoside
\ifodd\c@page\else
\null\newpage
\if@twocolumn\null\newpage\fi
\fi
\fi
\def\ps@chapterverso{\ps@empty}%
\newcommand{\tnorm}[1]{{\left\vert\kern-0.25ex\left\vert\kern-0.25ex\left\vert #1
    \right\vert\kern-0.25ex\right\vert\kern-0.25ex\right\vert}}

\numberwithin{equation}{section}
\usepackage{fancyhdr}
\newcommand{\BlackBox}{\rule{1.5ex}{1.5ex}}  % end of solutions

\makeatletter
\newcommand{\leqnomode}{\tagsleft@true}
\newcommand{\reqnomode}{\tagsleft@false}
\makeatother
\makeatletter
\newcommand{\lefteqno}{\let\veqno\@@leqno}
\makeatother
\title{\bf\large Theoretical and numerical study of the decay in a viscoelastic Bresse System}
\date{\vspace{-5ex}}
\author[1]{\small Jamilu Hashim Hassan}
\author[2,3]{\small Salim A. Messaoudi}
\author[4]{\small Toufic El-Arwadi}
\author[5]{Mohamad El Hindi}
\affil[1,3]{\small \it Department of Mathematics and Statistics, King Fahd University of Petroleum and Minerals\\ P.O. Box 546, Dhahran 31261, Saudi Arabia.}
\affil[2]{\small \it Department of Mathematics, University of Sharjah, P.O. Box 27272, Sharjah, United Arab Emirates.}
\affil[4,5]{\it Department of Mathematics and Computer Science,Beirut Arab University\\ P.O. Box 11-5020, Beirut, Lebanon.}
\affil[1]{\small elhashim06@yhaoo.com}
\affil[2]{\small smessaoudi@sharjah.ac.ae}
\affil[4]{\small t.elarwadi@bau.edu.lb}
\affil[5]{\small myh223@student.bau.edu.lb}
\begin{document}
\reqnomode
\maketitle

%% %% %% %% %% %% %% %% %% %% %% %% %% %% %% %% %% %% %% %% %% %% %% %% %% %% %% %% %% 
%=====================================================================================================================
%%%%%%%%%%%%%%%%%%%%%%%%%%%%%%%%%%%%%%%%%%%%%%%%%%%%%%%%%%%%%%%%%
\begin{abstract}
In this paper, we consider a one-dimensional finite-memory Bresse system with homogeneous Dirichlet-Neumann-Neumann boundary conditions. We prove some general decay results for the energy associated with the system in the case of equal and non-equal speeds of wave propagation under appropriate conditions on the relaxation function. In addition, we show by giving an example that in the case of equal speeds of wave propagation and for certain polynomially decaying relaxation functions, our result gives an optimal decay rate in the sense that the decay rate of the system is exactly the same as that of the relaxation function considered.
\end{abstract}
%%%%%%%%%%%%%%%%%%%%%%%%%%%%%%%%%%%%%%%%%%%%%%%%%%%%%%%%%%%%%%%%%

%%%%%%%%%%%%%%%%%%%%%%%%%%%%%%%%%%%%%%%%%%%%%%%%%%%%%%%%%%%%%%%%%%%%%%%%%%%%%%%%%%%%%%%%%%%%%%%%%%%%%%%%%%%%%%%%%%%%%%%%%%%%%

\section{\small{Introduction}}
\label{sec1}
Bresse system is a mathematical model that describes the vibration of a planar, linear shearable curved beam. The model was first derived by Bresse \cite{Bresse1859} and it consists of three coupled wave equations given by
\begin{equation}\label{e1s1}
\begin{array}{ll}
\r_1\vp_{tt}-k_1(\vp_x+\psi+lw)_x-lk_3(w_x-\vp)+F_1=0& \mathrm{in\,\,\,}(0,L)\times(0,\infty),\\
\\
\r_2\psi_{tt}-k_2\psi_{xx}+k_1(\vp_x+\psi+lw)+F_2=0& \mathrm{in\,\,\,}(0,L)\times(0,\infty),\\
\\
\r_1w_{tt}-k_3(w_x-\vp)_x+lk_1(\vp_x+\psi+lw)+F_3=0& \mathrm{in\,\,\,}(0,L)\times(0,\infty),
\end{array}
\end{equation}
where $\vp,\psi,w$ represent the vertical displacement, the shear angle, and the longitudinal displacement, respectively; $\r_1,\r_2,k_1,k_2,k_3,l$ are positive parameters and $F_1,F_2,F_3$ are external forces.

\par A lot of results dealing with well-posedness and asymptotic behaviour of the above system have been published. We start with the work of Santos $et \ al.$ \cite{Santos2010} from 2010, where they studied the Bresse system with Dirichlet-Dirichlet-Dirichlet boundary conditions and linear frictional damping acting on each equation, that is,
\begin{equation}\label{e2s1}
(F_1,F_2,F_3)=(\g_1\vp_t,\,\g_2\psi_t,\,\g_3w_t),
\end{equation}
where $\g_1,\g_2,\g_3>0$. They established an exponential decay rate for the system using spectral theory approach developed by Z. Liu and S. Zheng in \cite{Liu1999}. They also gave a numerical scheme using finite difference method to illustrate their theoretical result. Soriano $et \ al.$ \cite{Soriano2014} used the method developed by Lasiecka and Tataru in \cite{Lasiecka1993} and proved a uniform decay rate for the same system with a nonlinear frictional damping acting on the second equation and locally distributed nonlinear damping acting on the other equations. Precisely, the external forces are given by \[(F_1,F_2,F_3)=(\a(x)g_1(\vp_t),\, g_2(\psi_t),\, \g(x)g_3(w_t))\] with $\a,\g\in L^\infty(0,L)$ and the $g_i$'s are continuous and monotone increasing functions. The results of \cite{Santos2010} and \cite{Soriano2014} were established  without imposing any restriction on the speeds of wave propagation given by 
\begin{equation}\label{e3s1}
s_1=\sqrt{\frac{k_1}{\r_1}},\qquad s_2=\sqrt{\frac{k_2}{\r_2}},\qquad\mathrm{and}\qquad s_3=\sqrt{\frac{k_3}{\r_1}}.
\end{equation}
Alves $et \ al.$ \cite{Alves2015a} used the semigroup and spectral theory to obtain the exponential stability of the Bresse system with three controls at the boundary.
\par In the presence of dissipating terms in only one or two of the equations in system \eqref{e1s1}, the decay rates of the energy associated to the system depend totally on the speeds of the wave propagation. As illustrated in \cite{Alabau-Boussouira2011a}, Alabau-Boussouira $et \ al.$ studied \eqref{e1s1} with linear frictional damping acting on the second equation; that is, they used \eqref{e2s1}, with $\g_1=\g_3=0$ and $\g_2>0$ and showed that the system is exponentially stable if and only if it has equal speeds of wave propagation,
\begin{equation}\label{e4s1}
\frac{k_1}{\r_1}=\frac{k_2}{\r_2}=\frac{k_3}{\r_3}.
\end{equation}
As mentioned by many authors \cite{Alabau-Boussouira2011a, Alves2015}, relation \eqref{e4s1} is physically unrealistic. In the case of non-equal speeds of wave propagation, they proved polynomial stability with rates which can be improved with the regularity of the initial data. Fatori and Monteiro \cite{Fatori2012} improved this result in the case of non-equal speeds of wave propagation by proving optimal decay rate. Soriano $et \ al.$ \cite{Soriano2012} established the same exponential stability result as in \cite{Alabau-Boussouira2011a} by replacing the frictional damping with indefinite one; that is, they replaced $\g_2$ in \cite{Alabau-Boussouira2011a} with a function $a:(0,L)\longrightarrow\R$ such that $\displaystyle\bar a=\frac{1}{L}\int_0^La(x)dx>0$ and $\displaystyle\|a-\bar a\|_{L^2(0,L)}$ is small enough. Wehbe and Youcef \cite{Wehbe2010} inspected the situation of two locally distributed dampings acting on the last two equations; that is, \[(F_1,F_2,F_3)=(0,\, a_1(x)\psi_t,\, a_2(x)w_t),\] where $a_i:(0,L)\longrightarrow\R$ are non-negative functions which can take value zero on some part of the interval $(0,L)$. By using the frequency domain and the multiplier method, they  proved that the system is exponentially stable if and only if $s_1=s_2$. When $s_1\neq s_2$ they established a polynomial decay rate which can be improved with the regularity of the initial data. The same result was established by Alves $et \ al.$ in \cite{Alves2015}, in the case of non-equal speeds of wave propagation, they used the recent result of Borichev and Tomilov in \cite{Borichev2010} to show that the solution is polynomially stable with optimal decay rate.
\par Concerning the dissipation via heat effect, we mention the work of Liu and Rao \cite{Liu2009c} where the following system
\begin{equation}\label{e5s1}
\begin{array}{ll}
\r_1\vp_{tt}-k_1(\vp_x+\psi+lw)_x-lk_3(w_x-\vp)+l\g\chi=0& \mathrm{in\,\,\,}(0,L)\times(0,\infty),\\
\\
\r_2\psi_{tt}-k_2\psi_{xx}+k_1(\vp_x+\psi+lw)+\g\t_x=0& \mathrm{in\,\,\,}(0,L)\times(0,\infty),\\
\\
\r_1w_{tt}-k_3(w_x-\vp)_x+lk_1(\vp_x+\psi+lw)+\g\chi_t=0& \mathrm{in\,\,\,}(0,L)\times(0,\infty),\\
\\
\r_3\t_t-\t_{xx}+\g\psi_{xt}=0& \mathrm{in\,\,\,}(0,L)\times(0,\infty),\\
\\
\r_3\chi_t-\chi_{xx}+\g(w_x-l\vp)_t=0& \mathrm{in\,\,\,}(0,L)\times(0,\infty),
\end{array}
\end{equation}
with boundary and initial conditions was considered. They showed that the exponential stability of the system is equivalent to the validity of the identity \eqref{e4s1}. In the case where \eqref{e4s1} does not hold, they established a polynomial-type decay rate. Fatori and Mu{\~{n}}oz Rivera \cite{Fatori2010} obtained a similar result as in \cite{Liu2009c} for the thermoelastic Bresse system \eqref{e5s1} when the fifth equation is omitted. They also showed that the polynomial decay rate is optimal in the case of non-equal speeds of wave propagation. Filippo Dell'Oro \cite{DellOro2015} gave a detail stability analysis of the thermoelastic Bresse-Gurtin-Pipkin system of the form:
\begin{equation}\label{e6s1}
\begin{array}{ll}
\r_1\vp_{tt}-k(\vp_x+\psi+lw)_x-lk_0(w_x-\vp)=0& \mathrm{in\,\,\,}(0,L)\times(0,\infty),\\
\\
\r_2\psi_{tt}-k_2\psi_{xx}+k(\vp_x+\psi+lw)+\g\t_x=0& \mathrm{in\,\,\,}(0,L)\times(0,\infty),\\
\\
\r_1w_{tt}-k_0(w_x-\vp)_x+lk(\vp_x+\psi+lw)=0& \mathrm{in\,\,\,}(0,L)\times(0,\infty),\\
\\
\r_3\t_t-k_1\displaystyle\int_0^\infty g(s)\t_{xx}(t-s)ds+\g\psi_{xt}=0& \mathrm{in\,\,\,}(0,L)\times(0,\infty),
\end{array}
\end{equation}
where $g$ is a bounded convex integrable function on $[0,\infty)$ satisfying \[\int_0^\infty g(s)ds=1,\] and there exists a non-increasing absolutely continuous function $\mu:(0,\infty)\longrightarrow[0,\infty)$ such that
\[\mu(0)=\lim_{s\rightarrow0}\mu(s)\in(0,\infty),\qquad g(s)=\int_s^\infty\mu(\tau)d\tau,\qquad\forall \,s\in[0,\infty)\]
and
\[\mu'(s)+\nu\mu(s)\leq0 \quad\mathrm{for\ \ some\ \ }\nu>0\quad\mathrm{and\quad} a.e.\ s\in(0,\infty).\]
By introducing a new stability number of the form
\[\chi_g=\left(\frac{\r_1}{\r_3k}-\frac{1}{g(0)k_1}\right)\left(\frac{\r_1}{k}-\frac{\r_2}{b}\right)-\frac{1}{g(0)k}\frac{\r_1\g^2}{\r_3bk},\]
he proved that the semigroup generated by \eqref{e6s1} is exponentially stable if and only if
\[\chi_g=0\qquad\mathrm{and\qquad}k=k_0.\]
As a special case, he showed that his stability result gave the  stability characterization of Bresse systems with Fourier, Maxwell-Cataneo and Coleman-Gurtin thermal dissipation. The reader is referred to  \cite{Afilal2016, Gallego2017, Keddi2016, Najdi2014, Qin2014, Said-Houari2015, Said-Houari2016, Said-Houari2015a} and the references therein for more recent results on thermoelastic Bresse system.
\par There are few results that dealt with stabilization of Bresse system via infinite memory. We begin with the work of Guesmia and Kafini \cite{Guesmia2015} in 2015. They studied the following system
\begin{equation}\label{p1s1}
\begin{cases}
\r_1\vp_{tt}-k_1(\vp_x+\psi+lw)_x-lk_3(w_x-l\vp)+\displaystyle\int_0^\infty g_1(s)\vp_{xx}(x,t-s)ds=0,\\
\r_2\psi_{tt}-k_2\psi_{xx}+k_1(\vp_x+\psi+lw)+\displaystyle\int_0^\infty g_2(s)\psi_{xx}(x,t-s)ds=0,\\
\r_1w_{tt}-lk_3(w_x-l\vp)_x+lk_1(\vp_x+\psi+lw)+\displaystyle\int_0^\infty g_3(s)w_{xx}(x,t-s)ds=0,\\
\vp(0,t)=\psi(0,t)=w(0,t)=\vp(L,t)=\psi(L,t)=w(L,t)=0,\\
\vp(x,-t)=\vp_0(x,t),\ \vp_t(x,0)=\vp_1(x),\\
\psi(x,-t)=\psi_0(x,t),\ \psi_t(x,0)=\psi_1(x),\\
w(x,-t)=w_0(x,t),\ w_t(x,0)=w_1(x),
\end{cases}
\end{equation}
where $(x,t)\in(0,L)\times\R_+$, $g_i:\R_+\longrightarrow\R_+$ are differentiable non-increasing and integrable functions, and $L,\,l_i,\,\r_i,\,k_i$ are positive constants. They proved the well-posedness and the asymptotic stability of \eqref{p1s1}. Later, Guesmia and Kirane \cite{Guesmia2016} used two infinite memories to obtain the same stability result of \cite{Guesmia2015} under the following conditions on the speeds of wave propagation:
\begin{equation*}\label{e9s1}
\frac{k_1}{\r_1}=\frac{k_2}{\r_2}\quad\mathrm{in\ case\ } g_1=0,\quad \frac{k_1}{\r_1}=\frac{k_2}{\r_2}\quad\mathrm{in\ case\ } g_2=0,\quad \frac{k_1}{\r_1}=\frac{k_3}{\r_3}\quad\mathrm{in\ case\ } g_3=0.
\end{equation*}
Santos $et \ al.$ \cite{DeLimaSantos2015} discussed the Bresse system with only one infinite memory acting on the shear angle displacement equation. Precisely,  they studied problem \eqref{p1s1} with 
\[g_1=g_3=0\qquad\mathrm{and\qquad} g_2\mathrm{\quad satisfying:\quad} -\a_1 g_2(t)\leq g_2'(t)\leq -\a_2g_2(t),\quad\forall\,t\geq0,\]
for some $\a_1,\a_2>0.$ They showed that the solution of the system decays exponentially to zero if and only if \eqref{e4s1} holds, otherwise a polynomial stability of the system with an optimal decay rate of type $t^{-1/2}$ was obtained. Recently, Guesmia \cite{Guesmia2017} analysed the asymptotic stability of Bresse system with one infinite memory in the longitudinal displacement.
\par \textbf{To the best of our knowledge, there is no result in the literature that deals with the stability of Bresse system via viscoelastic damping of finite memory-type}. In this paper we will discuss the decay property of the following finite memory-type Bresse system:
\begin{equation}\label{p1}
\begin{cases}
\r_1\vp_{tt}-k_1(\vp_x+\psi+lw)_x-lk_3(w_x-l\vp) =0,&\mathrm{in}\,\,\,(0,L)\times(0,+\infty),\\
\r_2\psi_{tt}-k_2\psi_{xx}+k_1(\vp_x+\psi+lw)+\displaystyle\int_0^tg(t-s)\psi_{xx}(s)ds=0,&\mathrm{in\,\,\,}(0,L)\times(0,+\infty),\\
\r_1w_{tt}-k_3(w_x-l\vp)_x+lk_1(\vp_x+\psi+lw)=0,&\mathrm{in\,\,\,}(0,L)\times(0,+\infty),\\
\vp(0,t)=\vp(L,t)=\psi_x(0,t)=\psi_x(L,t)=w_x(0,t)=w_x(L,t)=0,&\mathrm{for\,\,\,}t\geq0,\\
\vp(x,0)=\vp_0(x),\,\,\,\,\vp_t(x,0)=\vp_1(x),&\mathrm{for}\,\,\, x\in (0,L),\\
\psi(x,0)=\psi_0(x),\,\,\,\,\psi_t(x,0)=\psi_1(x),&\mathrm{for}\,\,\, x\in (0,L),\\
w(x,0)=w_0(x),\,\,\,\,w_t(x,0)=w_1(x),&\mathrm{for}\,\,\, x\in (0,L),
\end{cases}\tag{$P$}
\end{equation}
where $l,\,\,k_1,\,\,k_2,\,\,k_3,\,\,\rho_1,\,\,\rho_2$ are positive constants, $\vp_0,\,\,\vp_1,\,\,\psi_0,\,\,\psi_1,\,\,w_0,\,\,w_1$ are given data and $g$ is a relaxation function satisfying some conditions to be specified in the next section. Our problem is motivated by the following classical Bresse system
\[\begin{array}{lll}
\r_1\vp_{tt}-S_x-lN &=0&\mathrm{in}\quad (0,L)\times(0,\infty),\\
\r_2\psi_{tt}-M_x+S &=0&\mathrm{in}\quad (0,L)\times(0,\infty),\\
\r_1w_{tt}-N_x-lS &=0&\mathrm{in}\quad (0,L)\times(0,\infty),
\end{array}\]
where $t$ and $x$ represent the time and space variables, respectively, and $N$, $S$ and $M$ denote the axial force, the shear force and the bending moment given by
\[S=k_1(\vp_x+\psi+w),\quad M=k_2\psi_x-\int_0^tg(t-s)\psi_{x}(\cdot,s)ds,\quad N=k_3(w_x-\vp).\]
 \textbf{We will prove, under a smallness condition on $l$, generalized energy decay results for the system in the case of equal and different speeds of wave propagation}. This paper is organized as follows: in Section \ref{sec2}, we state some preliminary results. In Section \ref{sec3}, we state and prove some technical lemmas. The statement and proof of our main results are given in Sections \ref{sec4} and \ref{sec5}, while in Section \ref{sec6} we present some numerical illustrations to validate our results. Through out this  work we use $c$ to represent a generic positive constant, independent of $t$ but may depend on the initial data.

%==============================================================================================================================
%=============================================================================================================================

\section{\small{Preliminaries}}
\label{sec2}
In this section, we introduce our assumptions, present some useful lemmas and state the existence theorem.\\
\textbf{Assumptions:} We assume that the relaxation function $g$ satisfies the following hypotheses:
\begin{itemize}
\item[(A1)] $g:[0,\infty)\longrightarrow[0,\infty)$ is a non-increasing differentiable function such that
\[g(0)>0\qquad\mathrm{and}\qquad k_2-\int_0^{+\infty}g(s)ds>0.\]
\item[(A2)] There exists a non-increasing differentiable function $\xi\!:[0,\infty)\longrightarrow(0,\infty)$ and a constant $p$, with $1\leq p<\frac{3}{2}$, such that \[g'(t)\leq-\xi(t)g^p(t),\qquad \forall t\geq0.\]
\end{itemize}
\begin{lemma}\label{l1s2}
Assume that $g$ satisfies hypotheses $(A1)$ and $(A2)$. Then,
\[\int_0^{+\infty}\xi(t)g^{1-\s}(t)dt<+\infty,\qquad\forall\,0<\s<2-p.\]
\end{lemma}
\begin{proof}
From (A1), we have \[\lim_{t\rightarrow+\infty}g(t)=0.\]
Using (A2), we have
\begin{eqnarray*}
\int_0^{+\infty}\xi(t)g^{1-\s}(t)dt&=&\int_0^{+\infty}\xi(t)g^p(t)g^{1-\s-p}(t)dt\leq-\int_0^{+\infty}g'(t)g^{1-\s-p}(t)dt\\
&=&-\left[\frac{1}{2-\s-p}g^{2-\s-p}(t)\right]_{t=0}^{t=+\infty}<+\infty,
\end{eqnarray*}
since $\s<2-p$.
\end{proof}

\par Now, integrating both sides of the second and third equations in $\eqref{p1}$ over $(0,L)$ and using the boundary conditions, we get
\begin{equation}\label{}
\frac{d^2}{dt^2}\int_0^L\psi(x,t)dx+\frac{k_1}{\r_2}\int_0^L\psi(x,t)dx+\frac{lk_1}{\r_2}\int_0^Lw(x,t)dx=0\quad\forall t\geq0
\end{equation}
and
\begin{equation}\label{}
\frac{d^2}{dt^2}\int_0^Lw(x,t)dx+\frac{l^2k_1}{\r_1}\int_0^Lw(x,t)dx+\frac{lk_1}{\r_1}\int_0^L\psi(x,t)dx=0\quad\forall t\geq0.
\end{equation}
Solving these ODEs simultaneously yields
\begin{equation}\label{}
\int_0^L\psi(x,t)dx=a_1\cos(a_0t)+a_2\sin(a_0t)+a_3t+a_4
\end{equation}
and
\begin{equation}\label{}
\int_0^Lw(x,t)dx=\frac{a_1}{l}\left(\frac{\r_2a_0^2}{k_1}-1\right)\cos(a_0t)+\frac{a_2}{l}\left(\frac{\r_2a_0^2}{k_1}-1\right)\sin(a_0t)-\frac{a_3}{l}t-\frac{a_4}{l},
\end{equation}
where
\[
\begin{cases}
a_0=\displaystyle{\sqrt{\frac{k_1}{\r_2}+\frac{l^2k_1}{\r_1}}}\\
\\
a_1=\displaystyle{\frac{k_1}{\r_2a_0^2}\int_0^L\psi_0(x)dx+\frac{lk_1}{\r_2a_0^2}\int_0^Lw_0(x)dx},\\
\\
a_2=\displaystyle{\frac{k_1}{\r_2a_0^3}\int_0^L\psi_1(x)dx+\frac{lk_1}{\r_2a_0^3}\int_0^Lw_1(x)dx},\\
\\
a_3=\displaystyle{\left(1-\frac{k_1}{\r_2a_0^2}\right)\int_0^L\psi_1(x)dx-\frac{lk_1}{\r_2a_0^2}\int_0^Lw_1(x)dx},\\
\\
a_4=\displaystyle{\left(1-\frac{k_1}{\r_2a_0^2}\right)\int_0^L\psi_0(x)dx+\frac{lk_1}{\r_2a_0^2}\int_0^Lw_0(x)dx}.
\end{cases}
\]
Therefore, we perform the following change of variables
\[
\begin{array}{l}
\displaystyle{\widetilde\psi =\psi-\frac{1}{L}\big(a_1\cos(a_0t)+a_2\sin(a_0t)+a_3t+a_4\big)}\\
\\
\displaystyle{\widetilde w= w-\frac{1}{L}\left[\frac{a_1}{l}\left(\frac{\r_2a_0^2}{k_1}-1\right)\cos(a_0t)+\frac{a_2}{l}\left(\frac{\r_2a_0^2}{k_1}-1\right)\sin(a_0t)-\frac{a_3}{l}t-\frac{a_4}{l}\right]}
\end{array}
\]
to get \[\int_0^L\widetilde\psi(x,t)dx=\int_0^L\widetilde w(x,t)dx=0, \qquad\forall\,t\geq0.\]
Furthermore, $(\vp,\,\widetilde\psi,\,\widetilde w)$ satisfies the equations and the boundary conditions in \eqref{p1} with the initial data
\[
\begin{array}{ll}
\displaystyle{\widetilde\psi_0=\psi_0-\frac{1}{L}(a_1+a_4)},&{\displaystyle\widetilde\psi_1=\psi_1-\frac{1}{L}(a_0a_2+a_3)}\\
\\
\displaystyle{\widetilde w_0=w_0-\frac{1}{L}\left[\frac{a_1}{l}\left(\frac{\r_2a_0^2}{k_1}-1\right)-\frac{a_4}{l}\right]},&\displaystyle{\widetilde w_1=w_1-\frac{1}{L}\left[\frac{a_2a_0}{l}\left(\frac{\r_2a_0^2}{k_1}-1\right)-\frac{a_3}{l}\right].}
\end{array}
\]
From now on, we work with $\widetilde\psi,\,\,\widetilde w$  and, respectively, write $\psi,\,\,w$ for convenience. We also introduce the following spaces,\[L^2_*(0,L):=\left\lbrace w\in L^2(0,L):\int_0^Lw(x)dx=0\right\rbrace,\qquad H^1_*(0,L):= H^1(0,L)\cap L^2_*(0,L),\] and \[H^2_*(0,L):= \left\lbrace w\in H^2(0,L):w_x(0)=w_x(L)=0\right\rbrace.\]
Then, Poincar\'e's inequality is applicable to the elements of $H^1_*(0,L)$, that is,
\begin{equation}\label{}
\exists\,\,c_0>0\qquad\mathrm{such\,\,\,\,that}\qquad \int_0^Lv^2dx\leq c_0\int_0^Lv_x^2dx\qquad\forall\,v\in H^1_*(0,L).
\end{equation}

For completeness, we state, without proof the global existence and regularity result which can be established by repeating the steps of the proof of the existence result in \cite{Messaoudi2016b}.
\begin{theorem}\label{t1s2}
Let $(\vp_0,\vp_1)\in H^1_0(0,L)\times L^2(0,L)$ and $(\psi_0,\psi_1),\,(w_0,w_1)\in H^1_*(0,L)\times L^2_*(0,L)$ be given. Assume that $g$ satisfies hypothesis $(A1)$. Then, the problem \eqref{p1} has a unique global (weak) solution
\[\vp\in C(\R_+;H^1_0(0,L))\cap C^1(\R_+;L^2(0,L)),\quad \psi,\,\,w\in C(\R_+;H^1_*(0,L))\cap C^1(\R_+;L^2_*(0,L)).\] Moreover, if \[(\vp_0,\vp_1)\in (H^2(0,L)\cap H^1_0(0,L))\times H^1_0(0,L)\] and \[(\psi_0,\psi_1),\,\,(w_0,w_1)\in {(H^2_*(0,L)\cap H^1_*(0,L))\times H^1_*(0,L)},\] then \[\vp\in C(\R_+;H^2(0,L)\cap H^1_0(0,L))\cap C^1(\R_+;H^1_0(0,L))\cap C^2(\R_+;L^2(0,L)),\] and \[\psi,\,\,w\in C(\R_+;H^2_*(0,L)\cap H^1_*(0,L))\cap C^1(\R_+;H^1_*(0,L))\cap C^2(\R_+;L^2(0,L)).\]
\end{theorem}
Now, we introduce the energy functional
\begin{equation}\label{e1s2}
\begin{split}
E(t):= \,\,\,&\frac{1}{2}\int_0^L\left[\r_1\vp_t^2+\r_2\psi_t^2+\r_1w_t^2+\left(k_2-\int_0^tg(s)ds\right)\psi_x^2\right.\\
&+\left.\vphantom{\int_0^t}k_3(w_x-l\vp)^2+k_1(\vp_x+\psi+lw)^2\right]dx+\frac{1}{2}(g\circ\psi_x)(t),\quad\forall\,t\geq0,
\end{split}
\end{equation}
where for any $v\in L^2_{loc}([0,+\infty);L^2(0,L))$, \[(g\circ v)(t):=\int_0^L\int_0^tg(t-s)\big(v(t)-v(s)\big)^2dsdx.\]
By multiplying the equations in \eqref{p1} by $\vp_t,\,\psi_t,\,w_t$, respectively, integrating over $(0,L)$ and exploiting the boundary conditions we have the following lemma.
\begin{lemma}\label{l2s2}
Let $(\vp,\psi,w)$ be the weak solution of \eqref{p1}. Then,
\begin{equation}\label{e2s2}
E'(t)=-\frac{1}{2}g(t)\int_0^L\psi_x^2dx+\frac{1}{2}(g'\circ\psi_x)(t)\leq0,\qquad\forall t\geq0.
\end{equation}
\end{lemma}
From the Cauchy-Schwarz and Poicar\'e's inequalities we have the following lemma.
\begin{lemma}[\cite{Messaoudi2007}]\label{l3s2}
There exists a constant $c>0$ such that for any $v\in L^2_{loc}(\R_+;H^1_*(0,L))$, we have \[\int_0^L\left(\int_0^tg(t-s)(v(t)-v(s))ds\right)^2dx\leq c(g\circ v_x)(t),\quad\forall t\geq0.\]
\end{lemma}
\begin{lemma}[\cite{Messaoudi2007}]\label{l4s2}
Assume that conditions $(A1)$ and $(A2)$ hold and let $(\vp,\psi,w)$ be the weak solution of \eqref{p1}. Then, for any $0<\s<1$, we have \[g\circ\psi_x\leq c\left[\int_0^tg^{1-\s}(s)ds\right]^{\frac{p-1}{p+\s-1}}(g^p\circ\psi_x)^{\frac{\s}{p+\s-1}}.\] For $\s=\tfrac{1}{2}$, we obtain the following inequality
\begin{equation}\label{e3s2}
g\circ\psi_x\leq c\left(\int_0^tg^{1/2}(s)ds\right)^{\frac{2p-2}{2p-1}}(g^p\circ\psi_x)^{\frac{1}{2p-1}}.
\end{equation}
\end{lemma}
\begin{corollary}\label{c1s2}
Assume that $g$ satisfies $(A1)$, $(A2)$ and $(\vp,\psi,w)$ is the weak solution of \eqref{p1}. Then, \[\xi(t)(g\circ\psi_x)(t)\leq c(-E'(t))^{\frac{1}{2p-1}},\quad\forall t\geq0.\]
\end{corollary}
\begin{proof}
Multiplying both sides of the inequality \eqref{e3s2} by $\xi(t)$ and using Lemmas \ref{l1s2} and \ref{l2s2}, we get
\begin{eqnarray*}
\xi(t)(g\circ\psi_x)(t)&\leq& c\xi^{\frac{2p-2}{2p-1}}(t)\left(\int_0^tg^{1/2}(s)ds\right)^{\frac{2p-2}{2p-1}}(\xi g^p\circ\psi_x)^{\frac{1}{2p-1}}(t)\\
&\leq&c\left(\int_0^t\xi(s)g^{1/2}(s)ds\right)^{\frac{2p-2}{2p-1}}(-g'\circ\psi_x)^{\frac{1}{2p-1}}\leq c(-E'(t))^{\frac{1}{2p-1}}.
\end{eqnarray*}
\end{proof}
\begin{lemma}[Jensen's inequality]\label{l8s2}
Let $G:[a,b]\longrightarrow\R$ be a concave function. Assume that the functions $f:\O\longrightarrow[a,b]$ and $h:\O\longrightarrow\R$ are integrable such that $h(x)\geq0$, for any $x\in\O$ and $\displaystyle\int_\O h(x)dx=k>0$. Then,
\[\frac{1}{k}\int_\O G(f(x))h(x)dx\leq G\left(\frac{1}{k}\int_\O f(x)h(x)dx\right).\] In particular, for $\displaystyle G(y)=y^{\frac{1}{p}},\,\,y\geq0,\,\,p>1$, we have \[\frac{1}{k}\int_\O f^{1/p}(x)h(x)dx\leq\left(\frac{1}{k}\int_\O f(x)h(x)dx\right)^{1/p}.\]
\end{lemma}

%%%%%%%%%%%%%%%%%%%%%%%%%%%%%%%%%%%%%%%%%%%%%%%%%%%%%%%%%%%%%%%%%%%%%%%%%%%%%%%%%%%%%%%%%%%%%%%%%%%%%%%%%%%%%%%%%

\section{\small{Technical Lemmas}}
\label{sec3}
In this section, we state and prove some lemmas needed to establish our main results. All the computations are done for regular solutions but they still hold for weak and strong solutions by a density argument.
\begin{lemma}\label{l1s3}
Assume that conditions $(A1)$ and $(A2)$ hold. Then, the functional $I_1$ defined by
\[I_1(t):=-\r_2\int_0^L\psi_t\int_0^tg(t-s)(\psi(t)-\psi(s))dsdx\]
satisfies, along the solution of \eqref{p1}, the estimates
\begin{eqnarray}\label{e1s3}
I_1'(t)&\leq&-\r_2\left(\int_0^tg(s)ds-\d\right)\int_0^L\psi_t^2dx+\d \int_0^L(\vp_x+\psi+lw)^2dx\nonumber\\
&&+c\d\int_0^L\psi_x^2dx+\frac{c}{\d}(g\circ\psi_x-g'\circ\psi_x),\qquad\forall\,\d>0.
\end{eqnarray}
\end{lemma}
\begin{proof}
Differentiating $I_1$, using equations in \eqref{p1} and integrating by parts, we get
\begin{eqnarray*}
I_1'(t)&=&-\r_2\int_0^L\psi_t\int_0^tg'(t-s)(\psi(t)-\psi(s))dsdx-\r_2\left(\int_0^tg(s)ds\right)\int_0^L\psi_t^2dx\\
&&+k_2\int_0^L\psi_x\int_0^tg(t-s)(\psi_x(t)-\psi_x(s))dsdx\\
&&+k_1\int_0^L(\vp_x+\psi+lw)\int_0^tg(t-s)(\psi(t)-\psi(s))dsdx\\
&&-\int_0^L\left(\int_0^tg(t-s)\psi_x(s)ds\right)\left(\int_0^tg(t-s)(\psi_x(t)-\psi_x(s))ds\right)dx.
\end{eqnarray*}
Next, we estimate the terms on the right-hand side of the above equation.
\par Using Young's inequality and Lemma \ref{l3s2} for $(-g')$, we obtain, for any $\d>0$,
\[-\r_2\int_0^L\psi_t\int_0^tg'(t-s)(\psi(t)-\psi(s))dsdx\leq \d\r_2\int_0^L\psi_t^2dx-\frac{c}{\d}(g'\circ\psi_x).\]
Similarly, we have
\[k_2\int_0^L\psi_x\int_0^tg(t-s)(\psi_x(t)-\psi_x(s))dsdx\leq \d\int_0^L\psi_x^2+\frac{c}{\d}(g\circ\psi_x),\]
\[k_1\int_0^L(\vp_x+\psi+lw)\int_0^tg(t-s)(\psi(t)-\psi(s))dsdx\leq  k_1\d\int_0^L(\vp_x+\psi+lw)^2dx+\frac{c}{\d}(g\circ\psi_x),\]
and
%\begin{eqnarray*}
%\lefteqn{-\int_0^L\left(\int_0^tg(t-s)\psi_x(s)ds\right)\left(\int_0^tg(t-s)(\psi_x(t)-\psi_x(s))ds\right)dx}\\&\leq& c\d\int_0^L\psi_x^2dx+c\left(\d+\frac{1}{\d}\right)(g\circ\psi_x).
%\end{eqnarray*}
%\[-\int_0^L\left(\int_0^tg(t-s)\psi_x(s)ds\right)\left(\int_0^tg(t-s)(\psi_x(t)-\psi_x(s))ds\right)dx\leq c\d\int_0^L\psi_x^2dx+c\left(\d+\frac{1}{\d}\right)(g\circ\psi_x).\]
\begin{equation*}
\begin{split}
-\int_0^L\left(\int_0^tg(t-s)\psi_x(s)ds\right)\left(\int_0^tg(t-s)(\psi_x(t)-\psi_x(s))ds\right)dx
\leq& c\d\int_0^L\psi_x^2dx\\&+c\left(\d+\frac{1}{\d}\right)(g\circ\psi_x).
\end{split}
\end{equation*}
A combination of these estimates gives the desired result.
\end{proof}

\begin{lemma}\label{l2s3}
Assume that the hypotheses $(A1)$ and $(A2)$ hold. Then, for any $\ve_0,\,\d_1>0$, the functional $I_2$ defined by
\[I_2(t):=-\r_1k_3\int_0^L(w_x-l\varphi)\int_0^xw_t(y,t)dydx-\r_1k_1\int_0^L\varphi_t\int_0^x(\varphi_x+\psi+lw)(y,t)dydx\]
satisfies, along the solution of \eqref{p1}, the estimate
\begin{eqnarray}\label{e2s3}
I_2'(t)&\leq&k_1^2\int_0^L(\varphi_x+\psi+lw)^2dx-k_3^2\int_0^L(w_x-l\varphi)^2dx+\frac{c}{\ve_0}\int_0^L\psi_t^2dx\nonumber\\
&&+\left(\ve_0-\r_1k_1+\frac{l\r_1|k_3-k_1|\d_1}{2}\right)\int_0^L\varphi_t^2dx\\
&&+\r_1\left(k_3+\frac{c_0l|k_3-k_1|}{2\d_1}\right)\int_0^Lw_t^2dx.\nonumber
\end{eqnarray}
\end{lemma}
\begin{proof}
Differentiation of $I_2$, using equations in \eqref{p1} and integration by parts yield
\begin{eqnarray*}
I_2&=&\r_1k_3\int_0^Lw_t^2dx+l\r_1k_3\int_0^L\vp_t\int_0^xw_t(y,t)dydx-k_3^2\int_0^L(w_x-l\vp)^2dx\\
&&+k_1^2\int_0^L(\vp_x+\psi+lw)^2dx-\r_1k_1\int_0^L\vp_t^2dx-\r_1k_1\int_0^L\vp_t\int_0^x(\psi_t+lw_t)(y,t)dydx.
\end{eqnarray*}
Using Young's inequality, we get, for any $\ve_0,\,\d_1>0$,
\begin{eqnarray*}
I_2&\leq&k_1^2\int_0^L(\varphi_x+\psi+lw)^2dx-k_3^2\int_0^L(w_x-l\varphi)^2dx+\frac{c}{\ve_0}\int_0^L\psi_t^2dx\\
&&+\left(\ve_0-\r_1k_1+\frac{l\r_1|k_3-k_1|\d_1}{2}\right)\int_0^L\varphi_t^2dx+\r_1\left(k_3+\frac{c_0l|k_3-k_1|}{2\d_1}\right)\int_0^Lw_t^2dx.\\
\end{eqnarray*}
\end{proof}

\begin{lemma}\label{l3s3}
Under the conditions $(A1)$ and $(A2)$, the functional $I_3$ defined by 
\[I_3(t):=-\r_1\int_0^L(\vp_x+\psi+lw)w_tdx-\frac{k_3\r_1}{k_1}\int_0^L(w_x-l\vp)\vp_tdx\]
satisfies, along the solution of \eqref{p1} and for any $\ve_0>0$, the estimate
\begin{eqnarray}\label{e3s3}
I_3'(t)&\leq&lk_1\int_0^L(\varphi_x+\psi+lw)^2dx-\frac{lk_3^2}{k_1}\int_0^L(w_x-l\varphi)^2dx+\frac{c}{\ve_0}\int_0^L\psi_t^2dx\nonumber\\
&&+\frac{l\r_1k_3}{k_1}\int_0^L\varphi_t^2dx+(\ve_0-l\r_1)\int_0^Lw_t^2dx+\r_1\left(\frac{k_3}{k_1}-1\right)\int_0^L\vp_{xt}w_tdx.
\end{eqnarray}
\end{lemma}
\begin{proof}
Differentiating $I_3$, using equations in \eqref{p1} and integrating by parts, we have
\begin{eqnarray*}
I_3&=&-\r_1\int_0^L\psi_tw_tdx-l\r_1\int_0^Lw^2_tdx+lk_1\int_0^L(\vp_x+\psi+lw)^2dx\\
&&+\frac{l\r_1k_3}{k_1}\int_0^L\vp_t^2dx-\frac{lk_3^2}{k_1}\int_0^L(w_x-l\vp)^2dx+\r_1\left(\frac{k_3}{k_1}-1\right)\int_0^L\vp_{xt}w_tdx.
%&\leq&lk_1\int_0^L(\varphi_x+\psi+lw)^2dx-\frac{lk_3^2}{k_1}\int_0^L(w_x-l\varphi)^2dx+\frac{c}{\ve_0}\int_0^L\psi_t^2dx\\
%&&+\frac{l\r_1k_3}{k_1}\int_0^L\varphi_t^2dx+(\ve_0-l\r_1)\int_0^Lw_t^2dx+\r_1\left(\frac{k_3}{k_1}-1\right)\int_0^L\vp_{xt}w_tdx.
\end{eqnarray*}
Use of Young's inequality for the first term in the right-hand side gives \eqref{e3s3}.
\end{proof}

\begin{lemma}\label{l4s3}
Assume that  conditions $(A1)$ and $(A2)$ hold. Then for any $\d>0$, the  functional $I_4$ defined by 
\[I_4(t):=-\int_0^L(\r_1\vp\vp_t+\r_2\psi\psi_t+\r_1ww_t)dx\]
satisfies, along the solution of \eqref{p1}, the estimate
\begin{eqnarray}\label{e4s3}
I_4'(t)&\leq&-\int_0^L(\r_1\vp_t^2+\r_2\psi_t^2+\r_1w_t^2)dx+k_1\int_0^L(\vp_x+\psi+lw)^2dx\nonumber\\
&&+k_3\int_0^L(w_x-l\vp)^2dx+\left(k_2+\d-\int_0^tg(s)ds\right)\int_0^L\psi_x^2dx+\frac{c}{\d}(g\circ\psi_x).
\end{eqnarray}
\end{lemma}
\begin{proof}
Differentiation of  $I_4$, using equations of \eqref{p1} gives
\begin{eqnarray*}
I_4'(t)&=&-\int_0^L(\r_1\vp_t^2+\r_2\psi_t^2+\r_1w_t^2)dx+k_1\int_0^L(\vp_x+\psi+lw)^2dx+k_3\int_0^L(w_x-l\vp)^2dx\\
&&+\left(k_2-\int_0^tg(s)ds\right)\int_0^L\psi_x^2dx-\int_0^L\psi_x\int_0^tg(t-s)(\psi_x(t)-\psi_x(s))dsdx.
%&\leq&-\int_0^L(\r_1\vp_t^2+\r_2\psi_t^2+\r_1w_t^2)dx+k_1\int_0^L(\vp_x+\psi+lw)^2dx+k_3\int_0^L(w_x-\vp)^2dx\\
%&&+\left(k_2-\int_0^tg(s)ds+\d\right)\int_0^L\psi_x^2+\frac{c}{\d}(g\circ\psi_x).
\end{eqnarray*}
Repeating the above computations yields the desired result.
\end{proof}

\begin{lemma}\label{l5s3}
Assume that conditions $(A1)$ and $(A2)$ hold. Then for any $\d,\,\d_2>0$, the functional $I_5$ defined by
\[I_5(t):=-\r_2\int_0^L\psi_x\int_0^x\psi_t(y,t)dydx\]
satisfies, along the solution of \eqref{p1}, the estimate
\begin{eqnarray}\label{e5s3}
I_5'(t)&\leq&\r_2\int_0^L\psi_t^2dx+\left(\frac{k_1}{2\d_2}+\int_0^tg(s)ds+\d-k_2\right)\int_0^L\psi_x^2dx\nonumber\\
&&+\frac{c_0k_1\d_2}{2}\int_0^L(\vp_x+\psi+lw)^2dx+\frac{c}{\d}(g\circ\psi_x).
\end{eqnarray}
\end{lemma}
\begin{proof}
Using equations of \eqref{p1} and repeating similar computations as above, we arrive at
\begin{eqnarray*}
I_5'(t)&=&\r_2\int_0^L\psi_t^2dx-k_2\int_0^L\psi_x^2dx+k_1\int_0^L\psi_x\int_0^x(\vp_x+\psi+lw)(y,t)dydx\\
&&+\int_0^L\psi_x\int_0^tg(t-s)\psi_x(s)dsdx\\
&\leq&\r_2\int_0^L\psi_t^2dx+\left(\frac{k_1}{2\d_2}+\int_0^tg(s)ds+\d-k_2\right)\int_0^L\psi_x^2dx\\
&&+\frac{k_1\d_2}{2}\int_0^L\left(\int_0^x(\vp_x+\psi+lw)(y,t)dy\right)^2dx+\frac{c}{\d}(g\circ\psi_x).
%&\leq&\r_2\int_0^L\psi_t^2dx+\left(\frac{k_1}{2\d_2}+\int_0^tg(s)ds+\d-k_2\right)\int_0^L\psi_x^2dx\nonumber\\
%&&+\frac{c_0k_1\d_2}{2}\int_0^L(\vp_x+\psi+lw)^2dx+\frac{c}{\d}(g\circ\psi_x).
\end{eqnarray*}
Poincar\'e's inequality for the third term yields \eqref{e5s3}.
\end{proof}

\begin{lemma}\label{l6s3}
Assume that the hypotheses $(A1)$ and $(A2)$ hold. Then, for any $\ve_0,\,\ve_1,\,\ve_2,\,\d>0$, the functional $I_6$ defined by
\[I_6(t):=\r_2\int_0^L\psi_t(\vp_x+\psi+lw)dx+\frac{b\r_1}{k_1}\int_0^L\vp_t\psi_xdx-\frac{\r_1}{k_1}\int_0^L\vp_t\int_0^tg(t-s)\psi_x(s)dsdx\]
satisfies, along the solution of \eqref{p1}, the estimate
\begin{eqnarray}\label{e6s3}
I_6'(t)&\leq&-k_1\int_0^L(\vp_x+\psi+lw)^2dx+\left(\frac{lk_2k_3\ve_1}{2k_1}+\frac{lk_3\ve_2}{2k_1}\int_0^tg(s)ds+\d\right)\int_0^L(w_x-l\vp)^2dx  \nonumber\\
&&+\d\int_0^L\vp^2_tdx+\left(\frac{lk_2k_3}{2k_1\ve_1}+\frac{lk_3}{2k_1\ve_2}\int_0^tg(s)ds+\frac{c}{\d}g(t)\right)\int_0^L\psi_x^2dx\\
&&+\ve_0\int_0^Lw_t^2dx+\frac{c}{\ve_0}\int_0^L\psi_t^2dx+\frac{c}{\d}(g\circ\psi_x-g'\circ\psi_x)+\left(\frac{k_2\r_1}{k_1}-\r_2\right)\int_0^L\vp_t\psi_{xt}dx\nonumber.
\end{eqnarray}
\end{lemma}
\begin{proof}
Use of  equations of \eqref{p1} and integration by parts lead to
\begin{eqnarray*}
I'_6(t)&=&-k_1\int_0^L(\vp_x+\psi+lw)^2dx+\r_2\int_0^L\psi_t^2dx+l\r_2\int_0^L\psi_tw_tdx\\
&&+\frac{lk_2k_3}{k_1}\int_0^L(w_x-l\vp)\psi_xdx-\frac{lk_3}{k_1}\int_0^L(w_x-l\vp)\int_0^tg(t-s)\psi_x(s)dsdx\\
&&-\frac{\r_1}{k_1}g(t)\int_0^L\vp_t\psi_xdx+\frac{\r_1}{k_1}\int_0^L\vp_t\int_0^tg'(t-s)(\psi_x(t)-\psi_x(s))dsdx\\
&&+\left(\frac{k_2\r_1}{k_1}-\r_2\right)\int_0^L\vp_x\psi_{xt}dx.
\end{eqnarray*}
Next, we estimate the terms in the right-hand side of the above equation.
\par Exploiting Young's inequality, we get
\[l\r_2\int_0^L\psi_tw_tdx\leq\ve_0\int_0^Lw^2_tdx+\frac{c}{\ve_0}\int_0^L\psi_t^2dx,\qquad\forall\,\ve_0>0.\]
Using Young's inequality and Lemma \ref{e3s2}, we obtain, for any $\ve_1,\,\ve_2,\,\d>0$,
\begin{eqnarray*}
\lefteqn{\frac{lk_2k_3}{k_1}\int_0^L(w_x-l\vp)\psi_xdx-\frac{lk_3}{k_1}\int_0^L(w_x-l\vp)\int_0^tg(t-s)\psi_x(s)dsdx}\\
&=&\frac{lk_3}{k_1}\left(k_2-\int_0^tg(s)ds\right)\int_0^L(w_x-l\vp)\psi_xdx\\
&&+\frac{lk_3}{k_1}\int_0^L(w_x-l\vp)\int_0^tg(t-s)(\psi_x(t)-\psi_x(s))dsdx\\
&\leq&\left(\frac{lk_2k_3\ve_1}{2k_1}+\frac{lk_3\ve_2}{2k_1}\int_0^tg(s)ds+\d\right)\int_0^L(w_x-l\vp)^2dx\\
&&+\left(\frac{lk_2k_3}{2k_1\ve_1}+\frac{lk_3}{2k_1\ve_2}\int_0^tg(s)ds\right)\int_0^L\psi_x^2dx+\frac{c}{\d}(g\circ\psi_x)
\end{eqnarray*}
and
\begin{eqnarray*}
\lefteqn{-\frac{\r_1}{k_1}g(t)\int_0^L\vp_t\psi_xdx+\frac{\r_1}{k_1}\int_0^L\vp_t\int_0^tg'(t-s)(\psi_x(t)-\psi_x(s))dsdx}\\
&\leq&\d\int_0^L\vp_t^2dx+\frac{c}{\d}g(t)\int_0^L\psi_x^2dx-\frac{c}{\d}(g'\circ\psi_x).
\end{eqnarray*}
A combination of these estimates gives the desired result.

\end{proof}

%%%%%%%%%%%%%%%%%%%%%%%%%%%%%%%%%%%%%%%%%%%%%%%%%%%%%%%%%%%%%%%%%%%%%%%%%%%%%%%%%%%%%%%%%%%%%%%%%%%%%%%%%%%%%%%%%

\section{\small{General Decay Rates for Equal Speeds of Wave Propagation}}
\label{sec4}
In this section, we state and prove a general decay result under equal speeds of wave propagation condition. The exponential and polynomial decay results are only special cases.
\begin{theorem}\label{t1s4}
Let $(\vp_0,\vp_1)\in H^1_0(0,L)\times L^2(0,L)$ and $(\psi_0,\psi_1),\,(w_0,w_1)\in H^1_*(0,L)\times L^2_*(0,L)$. Assume that  $(A1)$ and $(A2)$ hold and that
\begin{equation}\label{e1's4}
\frac{k_1}{\r_1}=\frac{k_2}{\r_2}\qquad\mathrm{and}\qquad k_1=k_3.
\end{equation} Then for $l$ small enough and for any $t_0>0$, the solution of \eqref{p1} satisfies, for $t> t_0$,
\begin{equation}\label{e1s4}
E(t)\leq C\exp\left(-\l\int_{t_0}^t\xi(s)ds\right),\qquad\mathrm{for\,\,\,\,} p=1,
\end{equation}
and
\begin{equation}\label{e2s4}
E(t)\leq C\left(\frac{1}{1+\int_{t_0}^t\xi^{2p-1}(s)ds}\right)^{\frac{1}{2p-2}},\qquad \mathrm{for\,\,\,\,} 1<p<\frac{3}{2},
\end{equation}
where $C>0$ is a constant independent of $t$ but may depend on the initial data  and $\l>0$ is a constant independent of both $t$ and the initial data.
Moreover, if
\begin{equation}\label{e3s4}
\int_{t_0}^{+\infty}\left(\frac{1}{1+\int_{t_0}^t\xi^{2p-1}(s)ds}\right)^{\frac{1}{2p-2}}dt<+\infty,\qquad \mathrm{for\,\,\,\,} 1<p<\frac{3}{2},
\end{equation}
then
\begin{equation}\label{e4s4}
E(t)\leq C\left(\frac{1}{1+\int_{t_0}^t\xi^p(s)ds}\right)^{\frac{1}{p-1}},\qquad \mathrm{for\,\,\,\,} 1<p<\frac{3}{2}.
\end{equation}
\end{theorem}
\begin{remark}\label{r1s4}
Inequalities \eqref{e2s4} and \eqref{e3s4} together give
\[\int_0^{+\infty}E(t)dt<+\infty.\]
\end{remark}
\begin{remark}
The smallness condition on $l$ makes the Bresse system close to Timoshenko system and, hence, inherits some of its stability properties.
\end{remark}
\begin{proof}[Proof of Theorem \ref{t1s4}]
Define a functional $\calL$ by
\[\calL:=NE+\sum_{j=1}^{6}N_jI_j,\]
where $N,\,N_j>0$ for $j=1,2,\dots,6$ with $N_3=N_6=1$. Then from \eqref{e1s3} $-$ \eqref{e6s3} we have
\begin{eqnarray*}
\calL'(t)&\leq&\left[-\r_1(k_1N_2+N_4)+\frac{l\r_1|k_3-k_1|\d_1N_2}{2}+\frac{l\r_1k_3}{k_1}+\ve_0N_2+\d\right]\int_0^L\vp_t^2dx\\
&&+\left[-\r_2\left(N_1\int_0^tg(s)ds+N_4-N_5\right)+\r_2\d N_1+\frac{c}{\ve_0}(1+N_2)\right]\int_0^L\psi_t^2dx\\
&&+\left[-l\r_1+\r_1(k_3N_2-N_4)+\frac{c_0l\r_1|k_3-k_1|N_2}{2\d_1}+\ve_0\right]\int_0^Lw_t^2dx\\
&&+\left[(N_5-N_4)\int_0^tg(s)ds+k_2(N_4-N_5)+\frac{k_1N_5}{2\d_2}+\frac{lk_2k_3}{2k_1\ve_1}\right.\\
&&\left.+\frac{lk_3}{2k_1\ve_2}\int_0^tg(s)ds+\d(cN_1+N_4+N_5)+\frac{c}{\d}g(t)\right]\int_0^L\psi_x^2dx\\
&&+\left[-\frac{lk_3^2}{k_1}-k_3(k_3N_2-N_4)+\frac{lk_2k_3\ve_1}{2k_1}+\frac{lk_2k_3\ve_2}{2k_1}\int_0^tg(s)ds+\d\right]\int_0^L(w_x-l\vp)^2dx\\
&&+\left[-k_1\left(1-k_1N_2-l-N_4-\frac{c_0\d_2N_5}{2}\right)+\d N_1\right]\int_0^L(\vp_x+\psi+lw)^2dx\\
&&+\frac{c}{\d}(1+N_1+N_4+N_5)g\circ\psi_x-\frac{c}{\d}(1+N_1)g'\circ\psi_x+NE'(t)\\
&&+\left(\frac{k_2\r_1}{k_1}-\r_2\right)\int_0^L\vp_t\psi_{xt}dx+\r_1\left(\frac{k_3}{k_1}-1\right)\int_0^L\vp_{xt}w_tdx.
\end{eqnarray*}
By setting  $\d_1=1,\,\,\,N_4=k_3N_2,\,\,\,N_5=4k_3N_2,\,\,\,\d_2=\frac{k_1}{k_2-g_0},\,\,\,\ve_1=\frac{k_3}{k_2},\,\,$ and $\ve_2=\frac{k_3}{2g_0},\,\,$ where ${\displaystyle g_0=\int_0^\infty g(s)ds}$, we arrive at

\begin{eqnarray*}
\calL'(t)&\leq&-\r_1\left[(k_1+k_3)N_2-l\left(\frac{|k_3-k_1|}{2}N_2+\frac{k_3}{k_1}\right)\right]\int_0^L\vp_t^2dx\\
&&-\r_2\left(N_1\int_0^tg(s)ds-3k_3N_2\right)\int_0^L\psi_t^2dx-l\r_1\left(1-\frac{c_0|k_3-k_1|}{2}N_2\right)\int_0^Lw_t^2dx\\
&&-\left[(k_2-g_0)k_3N_2-\frac{l}{k_1}\left(\frac{k_2^2}{2}+g_0^2\right)\right]\int_0^L\psi_x^2dx-\frac{lk_3^2}{4k_1}\int_0^L(w_x-l\vp)^2dx\\
&&-k_1\left[1-\left(k_1+k_3+\frac{2c_0k_1k_3}{k_2-g_0}\right)N_2-l\right]\int_0^L(\vp_x+\psi+lw)^2dx\\
&&+(1+N_2)\ve_0\int_0^L(\vp_t^2+w_t^2)dx+\frac{c}{\ve_0}(1+N_2)\int_0^L\psi_t^2dx+NE'(t)\\
&&+\d\int_0^L\Big(\vp_t^2+\r_2N_1\psi_t^2+c(N_1+5k_3N_2)\psi_x^2+N_1(\vp_x+\psi+lw)^2\Big)dx\\
&&+\frac{c}{\d}(1+N_1+5k_3N_2)g\circ\psi_x+\frac{c}{\d}(1+N_1)\left[g(t)\int_0^L\psi_x^2dx-g'\circ\psi_x\right]\\
&&+\left(\frac{k_2\r_1}{k_1}-\r_2\right)\int_0^L\vp_t\psi_{xt}dx+\r_1\left(\frac{k_3}{k_1}-1\right)\int_0^L\vp_{xt}w_tdx.
\end{eqnarray*}
Now, we set $\ve_0=\frac{l\r_1}{2(1+N_2)}$, to get
\begin{eqnarray*}
\calL'(t)&\leq&-\r_1\left[(k_1+k_3)N_2-l\left(\frac{1}{2}+\frac{k_3}{k_1}+\frac{|k_3-k_1|}{2}N_2\right)\right]\int_0^L\vp_t^2dx\\
&&-\r_2\left(N_1\int_0^tg(s)ds-3k_3N_2-\frac{c(1+N_2)^2}{l\r_1\r_2}\right)\int_0^L\psi_t^2dx-\frac{lk_3^2}{4k_1}\int_0^L(w_x-l\vp)^2dx\\
&&-\frac{l\r_1}{2}\left(1-c_0|k_3-k_1|N_2\right)\int_0^Lw_t^2dx-\left[(k_2-g_0)k_3N_2-\frac{l}{k_1}\left(\frac{k_2^2}{2}+g_0^2\right)\right]\int_0^L\psi_x^2dx\\
&&-k_1\left[1-\left(k_1+k_3+\frac{2c_0k_1k_3}{k_2-g_0}\right)N_2-l\right]\int_0^L(\vp_x+\psi+lw)^2dx\\
&&+\d c_{N_1,N_2}E(t)+\left[N-\frac{c}{\d}(1+N_1)\right]E'(t)+\frac{c}{\d}(1+N_1+5k_3N_2)g\circ\psi_x\\
&&+\left(\frac{k_2\r_1}{k_1}-\r_2\right)\int_0^L\vp_t\psi_{xt}dx+\r_1\left(\frac{k_3}{k_1}-1\right)\int_0^L\vp_{xt}w_tdx.
\end{eqnarray*}
Fix $t_0>0$ and choose $N_2$ so small that
\[1-c_0|k_3-k_1|N_2>0\qquad\mathrm{and}\qquad 1-\left(k_1+k_3+\frac{2c_0k_1k_3}{k_2-g_0}\right)N_2>0.\]
Next, we select $l$ small enough so that
\[(k_1+k_3)N_2-l\left(\frac{1}{2}+\frac{k_3}{k_1}+\frac{|k_3-k_1|}{2}N_2\right)>0,\qquad (k_2-g_0)k_3N_2-\frac{l}{k_1}\left(\frac{k_2^2}{2}+g_0^2\right)>0,\]
and
\[1-\left(k_1+k_3+\frac{2c_0k_1k_3}{k_2-g_0}\right)N_2-l>0.\]
After that, we pick $N_1$ very large so that \[N_1\int_0^{t_0}g(s)ds-3k_3N_2-\frac{c(1+N_2)^2}{l\r_1\r_2}>0.\] Therefore, we have
\begin{eqnarray*}
\calL'(t)&\leq&-(\b-c\d) E(t)+\left(N-\frac{c}{\d}\right)E'(t)+\frac{c}{\d}g\circ\psi_x\\
&&+\left(\frac{k_2\r_1}{k_1}-\r_2\right)\int_0^L\vp_t\psi_{xt}dx+\r_1\left(\frac{k_3}{k_1}-1\right)\int_0^L\vp_{xt}w_tdx,\nonumber
\end{eqnarray*}
for some $\b>0$. At this point, we take $\displaystyle\d<\frac{\b}{c}$. Consequently, we obtain, for some $k>0$,
\begin{eqnarray}\label{e5s4}
\calL'(t)&\leq&-kE(t)+(N-c)E'(t)+c(g\circ\psi_x)\nonumber\\
&&+\left(\frac{k_2\r_1}{k_1}-\r_2\right)\int_0^L\vp_t\psi_{xt}dx+\r_1\left(\frac{k_3}{k_1}-1\right)\int_0^L\vp_{xt}w_tdx,\qquad\forall\,t\geq t_0.
\end{eqnarray}
Finally, we choose N so large that $N>c$ and $\calL\sim E$, therefore we have, $\forall\, t\geq t_0$,
\begin{equation}\label{e6s4}
\calL'(t)\leq-kE(t)+c(g\circ\psi_x)+\left(\frac{k_2\r_1}{k_1}-\r_2\right)\int_0^L\vp_t\psi_{xt}dx+\r_1\left(\frac{k_3}{k_1}-1\right)\int_0^L\vp_{xt}w_tdx.
\end{equation}
\par Note that from this point, the proof goes similarly as in \cite{Messaoudi2017}. But we will continue for the sake of completeness.
\par By recalling \eqref{e1's4} and multiplying both sides of \eqref{e6s4} by $\xi(t)$ and using Corollary \ref{c1s2}, we arrive at
\begin{equation}\label{e7s4}
\xi(t)\calL'(t)\leq -k\xi(t)E(t)+c\xi(t)(g\circ\psi_x)(t)\leq-k\xi(t)E(t)+c(-E'(t))^{\frac{1}{2p-1}},\qquad\forall\,t\geq t_0.
\end{equation}
\par For $p=1$, it follows from non-increasing property of $\xi$ and \eqref{e7s4} that
\[\big(\xi(t)\calL(t)+cE(t)\big)'\leq\xi(t)\calL'(t)+cE'(t)\leq-k\xi(t)E(t),\qquad\forall\,t\geq t_0.\]
Using the fact that $\calF=\xi\calL+cE\sim E$, there exists a $\l>0$ such that
\[\calF'(t)\leq-\l\xi(t)\calF(t),\qquad\forall\,t\geq t_0.\] 
A simple integration over $(t,t_0)$ leads to
\[E(t)\leq C\exp\left(-\l\int_{t_0}^t\xi(s)ds\right),\qquad\forall\,t\geq t_0.\]

\par For $1<p<\frac{3}{2}$, we multiply both sides of \eqref{e7s4} by $(\xi E)^\a(t)$, with $\a=2p-2$, to obtain
\[\xi^{\a+1}(t)E^\a(t)\calL'(t)\leq-k(\xi E)^{\a+1}(t)+c(\xi E)^\a(t)(-E'(t))^{\frac{1}{2p-1}}.\]
Applying  Young's inequality with $\displaystyle q=\frac{\a+1}{\a}$ and $q'=\a+1$, we get
\[\xi^{\a+1}(t)E^\a(t)\calL'(t)\leq -(k-c\g)(\xi E)^{\a+1}(t)-c_\gamma E'(t),\qquad\forall\g>0.\]
We choose $\g$ such that $\l_1:=k-c\g>0$ and use the non-increasing property of $\xi$ and $E$, to have
\[(\xi^{\a+1}E^\a\calL)'(t)\leq\xi^{\a+1}(t)E^\a(t)\calL'(t)\leq-\l_1(\xi E)^{\a+1}(t)-cE'(t),\]
this entails that
\[(\xi^{\a+1}E^\a\calL+cE)'(t)\leq-\l_1(\xi E)^{\a+1}(t).\]
Let $\calF=\xi^{\a+1}E^\a\calL+cE\sim E$, then
\[\calF'(t)\leq-\l\xi^{\a+1}(t)\calF^{\a+1}(t),\]
for some $\l>0$.
Integration over $(t_0,t)$ gives
\[E(t)\leq C\left(\frac{1}{1+\int_{t_0}^t\xi^{2p-1}(s)ds}\right)^{\frac{1}{2p-2}},\qquad\forall t\geq t_0.\]
This establishes \eqref{e2s4}.
\par To prove \eqref{e4s4}, we treat  \eqref{e7s4} as follows
\begin{eqnarray}\label{e8s4}
\xi(t)\calL'(t)&\leq&-k\xi(t)E(t)+c\xi(t)(g\circ\psi_x)(t)\nonumber\\
&\leq&-k\xi(t)E(t)+c\frac{\eta(t)}{\eta(t)}\int_0^t\big(\xi^p(s)g^p(s)\big)^{\frac{1}{p}}\|\psi_x(t)-\psi_x(t-s)\|^2_2ds,
\end{eqnarray}
for any $t\geq t_0$, where
\begin{eqnarray*}
\eta(t)&=&\int_0^t\|\psi_x(t)-\psi_x(t-s)\|^2_2ds\leq 2\int_0^t(\|\psi_x(t)\|^2_2+\|\psi_x(t-s)\|^2_2)ds\\
&\leq& 4\int_0^t(E(t)+E(t-s))ds\leq 8\int_0^tE(t-s)ds=8\int_0^tE(s)ds\\
&\leq& 8\int_0^{+\infty}E(s)ds<+\infty,
\end{eqnarray*}
by Remark \ref{r1s4}.
Applying Jensen's inequality to the second term in the right-hand side of \eqref{e8s4}, with $\displaystyle G(y)=y^{\frac{1}{p}},$ $y>0$, $f(s)=\xi^p(s)g^p(s)$ and $h(s)=\|\psi_x(t)-\psi_x(t-s)\|^2_2$, we obtain
\[\xi(t)\calL'(t)\leq-k\xi(t)E(t)+c\eta(t)\left(\frac{1}{\eta(t)}\int_0^t\xi^p(s)g^p(s)\|\psi_x(t)-\psi_x(t-s)\|^2_2ds\right)^{\frac{1}{p}},\]
where we assume that $\eta(t)>0$, otherwise we get, from \eqref{e6s4},
\[E(t)\leq C\exp(-kt),\qquad\forall t\geq t_0.\]
Therefore,
\begin{eqnarray*}
\xi(t)\calL'(t)&\leq&-k\xi(t)E(t)+c\eta^{\frac{p-1}{p}}(t)\left(\xi^{p-1}(0)\int_0^t\xi(s)g^p(s)\|\psi_x(t)-\psi_x(t-s)\|^2_2ds\right)^{\frac{1}{p}}\\
&\leq&-k\xi(t)E(t)+c(-g'\circ\psi_x)^{\frac{1}{p}}(t)\leq-k\xi(t)E(t)+c(-E'(t))^{\frac{1}{p}}.
\end{eqnarray*}
Multiplying both sides of the above inequality by $(\xi E)^\a(t)$, for $\a=p-1$, and repeating the above computations, we arrive at
\[E(t)\leq C\left(\frac{1}{1+\int_{t_0}^t\xi^p(s)ds}\right)^{\frac{1}{p-1}},\qquad\forall t> t_0,\]
which establishes \eqref{e4s4}.
\end{proof}
\begin{example}\label{ex1s4}
Let $\displaystyle g(t)=\frac{a}{(1+t)^q}$ with $q>2$, and $a>0$ is to be chosen so that $(A1)$ is satisfied. Then
\begin{equation*}
g'(t)=-a_0\left(\frac{a}{(1+t)^q}\right)^{\frac{q+1}{q}}=-\xi(t)g^p(t),
\end{equation*}
with $\displaystyle\xi(t)=a_0=\frac{q}{a^{1/q}}$ and $\displaystyle p=\frac{q+1}{q}<\frac{3}{2}$, we have, for any fixed $t_0>0$,
\[\int_{t_0}^{+\infty}\left(\frac{1}{1+\int_{t_0}^t\xi^{2p-1}(s)ds}\right)^{\frac{1}{2p-2}}dt=\int_{t_0}^{+\infty}\left(\frac{1}{1+c(t-t_0)}\right)^{\frac{1}{2p-2}}dt<+\infty.\] Therefore, inequality \eqref{e4s4} entails that there exists $C>0$ such that
\[E(t)\leq C\left(\frac{1}{1+\int_{t_0}^t\xi^p(s)ds}\right)^{\frac{1}{p-1}}=\frac{c}{(1+t)^{q}},\]
with the optimal decay rate $q$. For more examples, see \cite{Messaoudi2017}.
\end{example}

%===========================================================================================
%+++++++++++++++++++++++++++++++++++++++++++++++++++++++++++++++++++++++++++++++++++++
%===========================================================================================

%===========================================================================================
%+++++++++++++++++++++++++++++++++++++++++++++++++++++++++++++++++++++++++++++++++++++
%===========================================================================================
\section{\small{General Decay Rate for Different Speeds of Wave Propagation}}\label{sec5}
In this section, we state and prove a generalized decay result in the case of non-equal speeds of wave propagation. We start by differentiating both sides of the differential equations in \eqref{p1} with respect to $t$ and use the fact that
\begin{eqnarray*}
\frac{\partial}{\partial t}\left[\int_0^tg(t-s)\psi_{xx}(s)ds\right]&=&\frac{\partial}{\partial t}\left[\int_0^tg(s)\psi_{xx}(t-s)ds\right]=g(t)\psi_{xx}(0)+\int_0^tg(s)\psi_{xxt}(t-s)ds\\
&=&\int_0^tg(t-s)\psi_{xxt}(s)ds+g(t)\psi_{0xx},
\end{eqnarray*}
to obtain the following system
\begin{equation}\label{p2}
\begin{cases}
\r_1\vp_{ttt}-k_1(\vp_{xt}+\psi_t+lw_t)_x-lk_3(w_{xt}-l\vp_t) =0,\\
\r_2\psi_{ttt}-k_2\psi_{xxt}+k_1(\vp_{xt}+\psi_t+lw_t)+\displaystyle\int_0^tg(t-s)\psi_{xxt}(s)ds+g(t)\psi_{0xx}=0,\\
\r_1w_{ttt}-k_3(w_{xt}-l\vp_t)_x+lk_1(\vp_{xt}+\psi_t+lw_t)=0.
\end{cases}\tag{$P_*$}
\end{equation}
The energy functional associated to \eqref{p2} is given by
\begin{equation}
\begin{split}
E_*(t):= \,\,\,&\frac{1}{2}\int_0^L\left[\r_1\vp_{tt}^2+\r_2\psi_{tt}^2+\r_1w_{tt}^2+\left(k_2-\int_0^tg(s)ds\right)\psi_{xt}^2\right.\\
&+\left.\vphantom{\int_0^t}k_3(w_{xt}-l\vp_t)^2+k_1(\vp_{xt}+\psi_t+lw_t)^2\right]dx+\frac{1}{2}(g\circ\psi_{xt})(t),\quad\forall\,t\geq0,
\end{split}
\end{equation}
Using similar arguments as in \cite[Lemma~3.11]{Guesmia2013} we have the following result.
\begin{lemma}\label{l5s2}
Let $(\vp,\psi,w)$ be the strong solution of \eqref{p1}. Then, the energy of \eqref{p2} satisfies, for all $t\geq0$,
\begin{equation}\label{e5s2}
E_*'(t)=-\frac{1}{2}g(t)\int_0^L\psi_{xt}^2dx+\frac{1}{2}(g'\circ\psi_{xt})-g(t)\int_0^L\psi_{tt}\psi_{0xx}dx
\end{equation}
and
\begin{equation}\label{e6s2}
E_*(t)\leq c\left(E_*(0)+\int_0^L\psi_{0xx}^2dx\right).
\end{equation}
\end{lemma}
%By repeating the steps of \cite{Messaoudi2007} in proving Lemma \ref{l4s2} and using \eqref{e6s2}, one can easily have the following lemma.
\begin{lemma}\label{l6s2}
Assume that hypotheses $(A1)$ and $(A2)$ hold and let $(\vp,\psi,w)$ be the strong solution of \eqref{p1}. Then, for any $0<\sigma<1$, we have
\[g\circ\psi_{xt}\leq\left[c\left(E_*(0)+\int_0^L\psi_{0xx}^2dx\right)\int_0^tg^{1-\s}(s)ds\right]^{\frac{p-1}{p+\s-1}}\big(g^p\circ\psi_{xt}\big)^{\frac{\s}{p+\s-1}}.\]
In particular, for $\s=\frac{1}{2}$, we get the following inequality
\begin{equation}\label{e7s2}
g\circ\psi_{xt}\leq c\left(\int_0^tg^{1/2}(s)ds\right)^{\frac{2p-2}{2p-1}}\left(g^p\circ\psi_{xt}\right)^{\frac{1}{2p-1}}.
\end{equation}
\end{lemma}
\begin{proof}
By setting $\displaystyle r=\frac{p+\s-1}{p-1}$ and $\displaystyle q=\frac{(p-1)(1-\s)}{p+\s-1}$, we have $\displaystyle \frac{r}{r-1}=\frac{p+\s-1}{\s}$ and $\displaystyle 1-q=\frac{\s p}{p+\s-1}$. Then exploiting H\"older's inequality and \eqref{e6s2}, we obtain
\begin{eqnarray*}
g\circ\psi_{xt}&=&\int_0^L\int_0^tg(t-s)(\psi_{xt}(t)-\psi_{xt}(s))^2dsdx\\
&=&\int_0^L\int_0^t\left[g^q(t-s)(\psi_{xt}(t)-\psi_{xt}(s))^{\frac{2}{r}}\right]\left[g^{1-q}(t-s)(\psi_{xt}(t)-\psi_{xt}(s))^{\frac{2r-2}{r}}\right]dsdx\\
&\leq&\left[\int_0^L\int_0^tg^{qr}(t-s)(\psi_{xt}(t)-\psi_{xt}(s))^2dsdx\right]^{\frac{1}{r}}\\&&\times\left[\int_0^L\int_0^tg^{\frac{(1-q)r}{r-1}}(t-s)(\psi_{xt}(t)-\psi_{xt}(s))^2dsdx\right]^{\frac{r-1}{r}}\\
&\leq&\left[\int_0^L\int_0^tg^{1-\s}(t-s)(\psi_{xt}(t)-\psi_{xt}(s))^2dsdx\right]^{\frac{p-1}{p+\s-1}}\left(g^p\circ\psi_{xt}\right)^{\frac{\s}{p+\s-1}}\\
&\leq&\left[2\int_0^L\int_0^tg^{1-\s}(s)(\psi_{xt}^2(t)+\psi_{xt}^2(t-s))dsdx\right]^{\frac{p-1}{p+\s-1}}\left(g^p\circ\psi_{xt}\right)^{\frac{\s}{p+\s-1}}\\
&\leq&\left[\frac{4}{k_2-g_0}\int_0^tg^{1-\s}(s)(E_*(t)+E_*(t-s))ds\right]^{\frac{p-1}{p+\s-1}}\left(g^p\circ\psi_{xt}\right)^{\frac{\s}{p+\s-1}}\\
&\leq&\left[c\left(E_*(0)+\int_0^L\psi_{0xx}^2dx\right)\int_0^tg^{1-\s}(s)ds\right]^{\frac{p-1}{p+\s-1}}\left(g^p\circ\psi_{xt}\right)^{\frac{\s}{p+\s-1}}.
\end{eqnarray*}
For $\s=\frac{1}{2}$, we get \eqref{e7s2}. This completes the proof.
\end{proof}
\begin{corollary}\label{c2s2}
Assume that conditions $(A1)$ and $(A2)$ hold and let $(\vp,\psi,w)$ be the strong solution of \eqref{p1}. Then,
\[\xi(t)(g\circ\psi_{xt})(t)\leq c\big(-E_*'(t)+c_1g(t)\big)^{\frac{1}{2p-1}},\qquad\forall t\geq0,\]
for some positive constant $c_1$.
\end{corollary}
\begin{proof}
From equation \eqref{e5s2} and inequality \eqref{e6s2} we have
\begin{eqnarray}\label{e8s2}
0\leq-g'\circ\psi_{xt}&=&-2E_*'(t)-g(t)\int_0^L\psi_{xt}^2dx-2g(t)\int_0^L\psi_{tt}\psi_{0xx}dx\notag\\
&\leq&-2E_*'(t)-2g(t)\int_0^L\psi_{tt}\psi_{0xx}dx\notag\\
&\leq&-2E_*'(t)+g(t)\int_0^L(\psi_{tt}^2+\psi_{0xx}^2)dx\notag\\
&\leq&-2E_*'(t)+g(t)\left(\frac{2}{\r_1}E_*(t)+\int_0^L\psi_{0xx}^2dx\right)\notag\\
&\leq&c\left(-E_*'(t)+c_1g(t)\right),
\end{eqnarray}
for some positive constant $c_1$. Multiplication of both sides of \eqref{e7s2} by $\xi(t)$ and use of Lemma \ref{l1s2} and inequality \eqref{e8s2} give
\begin{eqnarray*}
\xi(t)(g\circ\psi_{xt})(t)&\leq&c\left(\xi(t)\int_0^tg^{1/2}(s)ds\right)^{\frac{2p-2}{2p-1}}\big(\xi g^p\circ\psi_{xt}\big)^{\frac{1}{2p-1}}(t)\\
&\leq&c\left(\int_0^t\xi(s)g^{1/2}(s)ds\right)^{\frac{2p-2}{2p-1}}\big(-g'\circ\psi_{xt}\big)^{\frac{1}{2p-1}}(t)\\
&\leq&c\big(-E_*'(t)+c_1g(t)\big)^{\frac{1}{2p-1}}.
\end{eqnarray*}
\end{proof}
Now we estimate the third term in the right-hand side of \eqref{e6s4} as in \cite{Guesmia2013}.
\begin{lemma}\label{l1s5}
Let $(\vp,\psi,w)$ be the strong solution of \eqref{p1}. Then, for any $\ve>0$, we have
\begin{equation}\label{e2s5}
\left(\frac{\r_1k_2}{k_1}-\r_2\right)\int_0^L\vp_t\psi_{xt}dx\leq \ve E(t)+\frac{c}{\ve}(g\circ\psi_{xt}-E'(t)+g(t)),\quad\forall t\geq t_0.
\end{equation}
\end{lemma}
\begin{proof}
\begin{eqnarray}\label{e1's5}
\left(\frac{\r_1k_2}{k_1}-\r_2\right)\int_0^L\vp_t\psi_{xt}dx&=&\frac{\left(\frac{\r_1k_2}{k_2}-\r_2\right)}{\int_0^tg(s)ds}\int_0^L\vp_t\int_0^tg(t-s)(\psi_{xt}(t)-\psi_{xt}(s))dsdx\nonumber\\
&&+\frac{\left(\frac{\r_1k_2}{k_1}-\r_2\right)}{\int_0^tg(s)ds}\int_0^L\vp_t\int_0^tg(t-s)\psi_{xt}(s)dsdx.
\end{eqnarray}
By observing that $\displaystyle \int_0^tg(s)ds\geq \int_0^{t_0}g(s)ds$, for all $t\geq t_0$ and exploiting Young's inequality and\\
\\
Lemma \ref{l3s2} (for $\psi_{xt}$), we get, for $\ve>0$ and $t\geq t_0$,
\[\frac{\left(\frac{\r_1k_2}{k_1}-\r_2\right)}{\int_0^tg(s)ds}\int_0^L\vp_t\int_0^tg(t-s)(\psi_{xt}(t)-\psi_{xt}(s))dsdx\leq\frac{\ve}{4}\r_1\int_0^L\vp_t^2dx+\frac{c}{\ve}(g\circ\psi_{xt}).\]
On the other hand, by integration by parts and using Lemma \ref{l3s2} (for $-g'$ and $\psi_x$) and the fact that $E$ is non-increasing, we get
\begin{eqnarray*}
\lefteqn{\frac{\left(\frac{\r_1k_2}{k_1}-\r_2\right)}{\int_0^tg(s)ds}\int_0^L\vp_t\int_0^tg(t-s)\psi_{xt}(s)dsdx}\\&=&\frac{\left(\frac{\r_1k_2}{k_1}-\r_2\right)}{\int_0^tg(s)ds}\int_0^L\vp_t\left(g(0)\psi_x-g(t)\psi_{0x}+\int_0^tg'(t-s)\psi_{x}(s)ds\right)dx\\
&=&\frac{\left(\frac{\r_1k_2}{k_1}-\r_2\right)}{\int_0^tg(s)ds}\int_0^L\vp_t\left(g(t)(\psi_x-\psi_{0x})-\int_0^tg'(t-s)(\psi_x(t)-\psi_{x}(s))ds\right)dx\\
&\leq&\frac{\ve}{4}\r_1\int_0^L\vp_t^2dx +\frac{c}{\ve}g(t)\int_0^L(\psi_x^2+\psi_{0x}^2)dx-\frac{c}{\ve}g'\circ\psi_x\\
&\leq&\frac{\ve}{4}\r_1\int_0^L\vp_t^2dx+\frac{c}{\ve}E(0)g(t)-\frac{c}{\ve}g'\circ\psi_x.
\end{eqnarray*}
Inserting the last two inequalities in \eqref{e1's5}, we get \eqref{e2s5}.
\end{proof}

\begin{comment}
\begin{lemma}\label{l2s5}
Let $(\vp,\psi,w)$ be the strong solution of \eqref{p1}. Then, for any $t\geq t_0$, we have
\begin{equation}\label{e3s5}
\mathscr L'(t)\leq-kE(t)+c(g\circ\psi_x+g\circ\psi_{xt})+c\left(E(0)+E_*(0)+\int_0^L\psi_{0xx}dx\right)g(t).
\end{equation}
\end{lemma}
\begin{proof}
It follows from Young's inequality and \eqref{e6s2} that
\begin{eqnarray}\label{e4s5}
-\int_0^L\psi_{tt}\psi_{0xx}dx&\leq&\frac{1}{2}\int_0^L(\psi_{tt}^2+\psi_{0xx}^2)dx\leq c\left (E_*(t)+\int_0^L\psi_{0xx}^2dx\right)\nonumber\\
&\leq& c\left(E_*(0)+\int_0^L\psi_{0xx}^2dx\right).
\end{eqnarray}
Then plugging \eqref{e2s5} and \eqref{e4s5} into \eqref{e1s5}, we get % and $N_1$ large enough such that $\displaystyle \frac{N_1}{2}-4cN_2^2-c-\frac{c}{\ve}>0$, we obtain \eqref{e3s5}.
\begin{eqnarray*}
\mathscr L'(t)&\leq&-\frac{K}{4}\int_0^L(\vp_x+\psi)^2dx-\left(\frac{lN_3}{2}-\frac{5}{4}c\right)\int_0^L\psi_x^2dx-\big(\tau-(N_3+1)\ve_0\big)\int_0^L\r_1\vp_t^2dx\\
&&-\left(N_2g_0-\frac{1}{4}-\frac{cN_3}{\ve_0}-c\right)\r_2\int_0^L\psi_t^2dx+c\left(4N_2^2+N_3+\frac{5}{4}\right)g\circ\psi_x+\frac{c}{\ve_0}g\circ\psi_{xt}\\
&&+\left(\frac{N_1}{2}-4cN_2^2-c-\frac{c}{\ve_0}\right)g'\circ\psi_x+\frac{c}{\ve_0}E(0)g(t)+c\left(E_*(0)+\int_0^L\psi_{0xx}^2dx\right)g(t).
\end{eqnarray*}
At this point, we choose $N_3,\,\ve_0,\,N_2$ and $N_1$ as in \eqref{e5s4} to get \eqref{e3s5}.
\end{proof}
\end{comment}

\begin{theorem}\label{t1s5}
Let \[(\vp_0,\vp_1)\in \left(H^2(0,L)\cap H^1_0(0,L)\right)\times H^1_0(0,L)\] and \[(\psi_0,\psi_1),\,\,(w_0,w_1)\in (H^2_*(0,L)\cap H^1_*(0,L))\times H^1_*(0,L).\]  Assume that conditions $(A1)$, $(A2)$ hold and that
\[\frac{\r_1}{k_1}\neq\frac{\r_2}{k_2}\qquad\mathrm{and}\qquad k_1=k_3.\]
Then for $l$ small enough and for any $t_0>0$, there exists a positive constant $C$ that may depend on the initial data but independent of $t$, for which the strong solution of \eqref{p1} satisfies, for $t> t_0$,
\begin{equation}\label{e5s5}
E(t)\leq C\left(\frac{1}{\int_{t_0}^t\xi^{2p-1}(s)ds}\right)^{\frac{1}{2p-1}},\qquad\mathrm{for\,\,\,\,}1\leq p<\frac{3}{2}.
\end{equation}
\end{theorem}
\begin{proof}
Repeating the steps of the proof of Theorem \ref{t1s4} up to inequality \eqref{e5s4}, then inserting \eqref{e2s5}  into \eqref{e5s4} we obtain
\[\calL'(t)\leq-(k-\ve)E(t)+\left[N-c\left(1+\frac{1}{\ve}\right)\right]E'(t)+cg\circ\psi_x+\frac{c}{\ve}g\circ\psi_{xt}+\frac{c}{\ve}g(t),\qquad\forall\,t\geq t_0.\]
Now we choose $\ve$ so small that $k-\ve>0$, and then pick $N>c\left(1+\frac{1}{\ve}\right)$ to get
\[\calL'(t)\leq-k_0E(t)+c(g\circ\psi_x+g\circ\psi_{xt})+cg(t),\qquad\forall\,t\geq t_0,\]
for some $k_0>0$.
We then multiply both sides of the above inequality by $\xi(t)$ and use Corollaries \ref{c1s2} and \ref{c2s2} to get
\begin{eqnarray*}
\xi(t)\calL'(t)&\leq&-k_0\xi(t)E(t)+c\xi(t)(g\circ\psi_x+g\circ\psi_{xt})+c\xi(t)g(t)\\
&\leq&-k_0\xi(t)E(t)+c\left[\big(-E'(t)\big)^{\frac{1}{2p-1}}+\big(-E_*'(t)+c_1g(t)\big)^{\frac{1}{2p-1}}\right]+c\xi(t)g(t).
\end{eqnarray*}
Next, we set $\a=2p-2$, then multiply both sides of the above inequality by $(\xi E)^\a(t)$ and exploit Young's inequality, with $\displaystyle q=\frac{\a+1}{\a}$ and $q'=\a+1$, to obtain
\[\xi^{\a+1}(t)E^\a(t)\calL'(t)\leq-(k_0-c\g)(\xi E)^{\a+1}(t)-cE'(t)-cE_*'(t)+c_1g(t)+c\xi^{\a+1}(t)E^\a(t)g(t),\,\,\,\,\forall \g>0.\]
We choose $\g>0$ so small such that $\l_2:=k_0-c\g>0$ and use the non-increasing property of $\xi$ and $g$ to get
\[(\xi^{\a+1}E^\a\calL+cE+cE_*)'(t)\leq-\l_2(\xi E)^{\a+1}(t)+c\xi^{\a+1}(t)E^\a(t)g(t)+c_1g(t),\]
which implies that
\[\l_2(\xi E)^{\a+1}(t)\leq -(\xi^{\a+1}E^\a\calL+cE+cE_*)'(t)+c\xi^{\a+1}(t)E^\a(t)g(t)+c_1g(t).\]
Then integration over $(t_0,t)$ together with the non-increasing property of $E$ and $\xi$, and the hypothesis $(A1)$ yield, for $t\geq t_0$,
\begin{eqnarray*}
\l_2E^{\a+1}(t)\int_{t_0}^t\xi^{\a+1}(s)ds&\leq&\l_2\int_{t_0}^t(\xi E)^{\a+1}(s)ds\leq-(\xi^{\a+1}E^\a\calL+cE+cE_*)(t)\\
&&+(\xi^{\a+1}E^\a\calL+cE+cE_*)(0)+\int_0^L\psi_{0xx}^2dx\\&&+(c\xi^{\a+1}(0)E^\a(0)+c_1)\int_{t_0}^tg(s)ds\\
&\leq&(\xi^{\a+1}E^\a\calL+cE+cE_*)(0)+\int_0^L\psi_{0xx}^2dx\\&&+(c\xi^{\a+1}(0)E^\a(0)+c_1)\int_0^\infty g(s)ds.
\end{eqnarray*}
Therefore, we get
\[E(t)\leq C\left(\frac{1}{\int_{t_0}^t\xi^{2p-1}(s)ds}\right)^{\frac{1}{2p-1}},\qquad\forall t> t_0.\]
This completes the proof of the Theorem \ref{t1s5}.
\end{proof}
%\begin{remark}
%Note that our estimate \eqref{e5s5} improves the one given in \cite{Guesmia2013}.
%\end{remark}
\begin{example}
Let $\displaystyle g(t)=e^{-at}$, where $a>0$. Then $g'(t)=-\xi(t)g(t)$ with $\xi(t)=a$. It follows from \eqref{e5s5} that for any fixed $t_0>0$, there exists $C>0$ such that \[E(t)\leq\frac{C}{t-t_0},\qquad\forall\,t> t_0.\]
\end{example}
\begin{example}
Consider the same function $g$ as in Example \ref{ex1s4} and write $g'$ as in Example \ref{ex1s4}. Then it follows from \eqref{e5s5} that for any fixed $t_0>0$, there exists $C>0$ such that
\[E(t)\leq C\left(\frac{1}{\int_{t_0}^t\xi^{2p-1}(s)ds}\right)^{\frac{1}{2p-1}}=\frac{c}{(1+t)^{\frac{q}{q+2}}}\qquad\forall\,t> t_0.\]
For more examples, see \cite{Guesmia2013}.
\end{example}

%===========================================================================================
%+++++++++++++++++++++++++++++++++++++++++++++++++++++++++++++++++++++++++++++++++++++
%===========================================================================================

\section{Full discrete problem}\label{sec6}

In this section, we introduce a scheme for the problem based on $P_{1}$%
-finite element method in space and implicit Euler scheme for time
discretization. Then we draw graphs for the discredited energy showing it's
decay in both cases, polynomial and exponential. Finally, we implement the
approximation of the solutions $\varphi ,$ $\psi $ and $w$ in $3D$ and their
cross section at $x=0.5$. 

\subsection{Finite element setup}

We denote by $(\Gamma _{h})_{h}$ a partition of $\Omega $ which fulfills the
following conditions:
\begin{enumerate}
\item $\Gamma _{h}=\{R\subset \bar{\Omega};$ $R$ is closed in $\Omega \}$;
\item $\forall (R,R^{\prime })\in \Gamma _{h}\times \Gamma _{h}; \ \left\vert
R\right\vert =\left\vert R^{\prime }\right\vert $, where their intersection
is either empty or an end point;
\item $\bar{\Omega}=\bigcup\limits_{R\in \Gamma _{h}}R$.
\end{enumerate}
We define the uniform partition of $\Omega $ as $0=x_{0}<x_{1}<\cdots <x_{s}$
and denote the length of $(x_{j},x_{j+1})$ as $h=\frac{L}{s}$. Now for
time discretization, denote by $\Delta t=\frac{T}{N}$ the step time, where $T$
is the total time and $N$ is a positive integer. Finally we define the discrete finite
element space by
\[ s_{h}=\{u_{h}\in H^{1}(0, L):\forall R\in \Gamma _{h};u_{h}|_{R}\in P_{1}(R)\}, \]
where $P_{k}(R)$ denotes the space of restrictions of $R$ of polynomials
with one variable and of order less than or equal to $k$.

Now we introduce the scheme and the discrete energy by using implicit Euler
scheme
\[
\begin{cases}
\frac{\rho _{1}}{\Delta t}(\Phi _{h}^{n}-\Phi _{h}^{n-1},\bar{\varphi}%
_{h})+ k _{1}(\varphi _{h,x}^{n}+\psi _{h}^{n}+lw_{h}^{n},\bar{\varphi}%
_{h,x})-lk _{3}(w_{h,x}^{n}-l\varphi _{h}^{n},\bar{\varphi}_{h})=0 \\ 
\frac{\rho _{2}}{\Delta t}(\Psi _{h}^{n}-\Psi _{h}^{n-1},\bar{\psi}%
_{h})+k _{2}(\psi _{h,x}^{n},\bar{\psi}_{x})+k _{1}(\varphi
_{h,x}^{n}+\psi _{h}^{n}++lw_{h}^{n},\bar{\psi}_{h}) \\ 
\qquad\qquad\qquad\qquad -\Delta t\sum\limits_{m=1}^{n}g(t_{n-m})(\psi _{h,x}^{m},\bar{\psi}_{h,x})=0
\\ 
\frac{\rho _{1}}{\Delta t}(W_{h}^{n}-W_{h}^{n-1},\bar{w}_{h})+k
_{3}(w_{h,x}^{n}-l\varphi _{h}^{n},\bar{w}_{h,x})+k _{1}l(\varphi
_{h,x}^{n}+\psi _{h}^{n}++lw_{h}^{n},\bar{w}_{h})=0%
\end{cases}
\] 
where $t_{j}=j\Delta t$ and
\begin{eqnarray*}
E^{n} &=& \rho _{1}\left\vert \left\vert \Phi _{h}^{n}\right\vert \right\vert
^{2}+\rho _{1}\left\vert \left\vert W_{h}^{n}\right\vert \right\vert
^{2}+\rho _{2}\left\vert \left\vert \Psi _{h}^{n}\right\vert \right\vert
^{2}+k _{1}\left\vert \left\vert \varphi _{h,x}^{n}+\psi
_{h}^{n}+lw_{h}^{n}\right\vert \right\vert ^{2} \\
&& + k _{3}\left\vert
\left\vert w_{h,x}^{n}-l\varphi _{h}^{n}\right\vert \right\vert ^{2}+k
_{2}\left\vert \left\vert \psi _{h,x}^{n}\right\vert \right\vert
^{2}- \left( \int\limits_{0}^{t_{n}}g(t) dt \right) \left\vert \left\vert \psi
_{h,x}^{n}\right\vert \right\vert ^{2} \\
&& +\frac{1}{2}\Delta
t\int\limits_{0}^{L}\sum\limits_{m=1}^{n}g(t_{n-m})(\psi _{h,x}^{n}-\psi
_{h,x}^{m})^{2}dx
\end{eqnarray*}
\subsection{Numerical Experiments}

By using the following data
\[ k_{1} = k_{2} = k_{3} = 1,\ \rho_{1} = \rho_{2} = 0.1, \ \Delta t = 0.012, \ h = 0.024, \ T=7.4 \ \ \mathrm{and} \ \ g(x)=e^{-3x}; \]
we draw the solutions $\varphi ,$ $\psi ,$ and $w$ in $3D$ (see Figures \ref{fig:1}, \ref{fig:2} and \ref{fig:3}, respectively) and their cross section at $x=0.5$ (see Figures \ref{fig:4}, \ref{fig:5} and \ref{fig:6}, respectively).
\begin{figure}
	\begin{center}
	\includegraphics[width=5in]{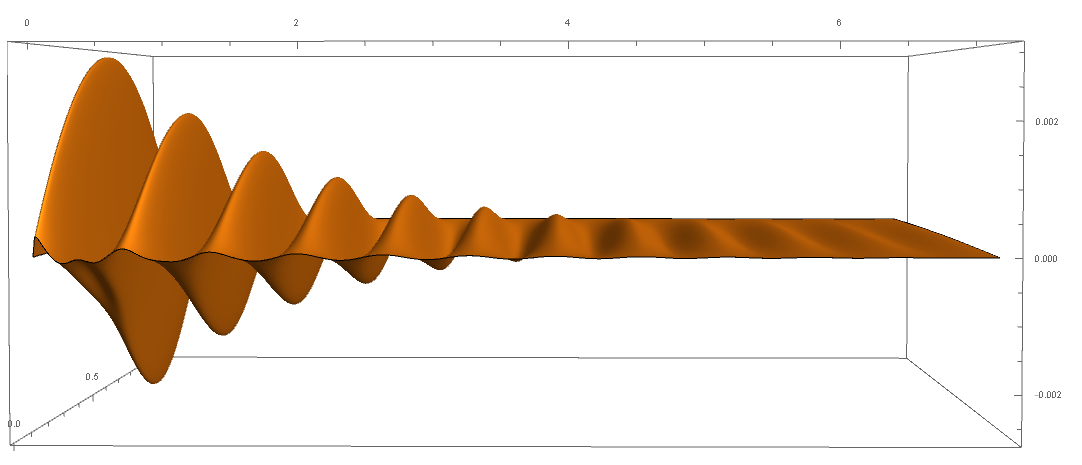}
		\caption{The evolution in time and space of $\varphi$}\label{fig:1}

	\end{center}
\end{figure}
\begin{figure}
	\begin{center}
		\includegraphics[width=5in]{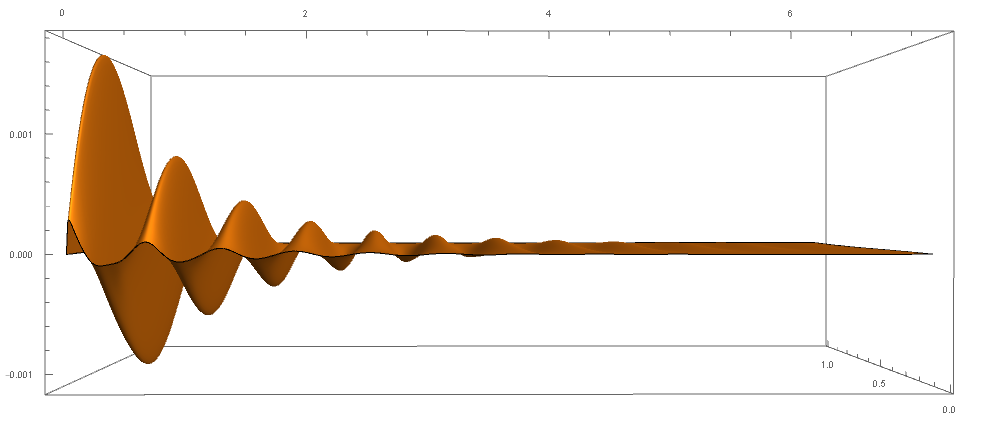}
		\caption{The evolution in time and space of $\psi$}\label{fig:2}

	\end{center}
\end{figure}
\begin{figure}
	\begin{center}
		\includegraphics[width=5in]{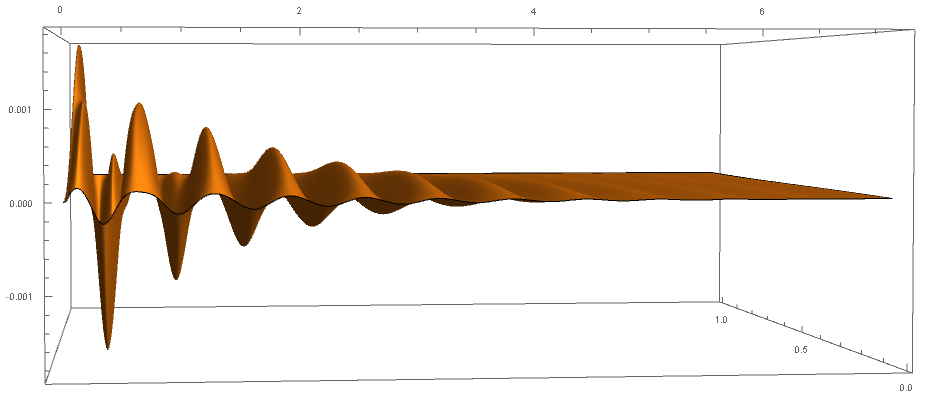}
		\caption{The evolution in time and space of $w$}\label{fig:3}

	\end{center}
\end{figure}
\begin{figure}
	\begin{center}
		\includegraphics[width=5in]{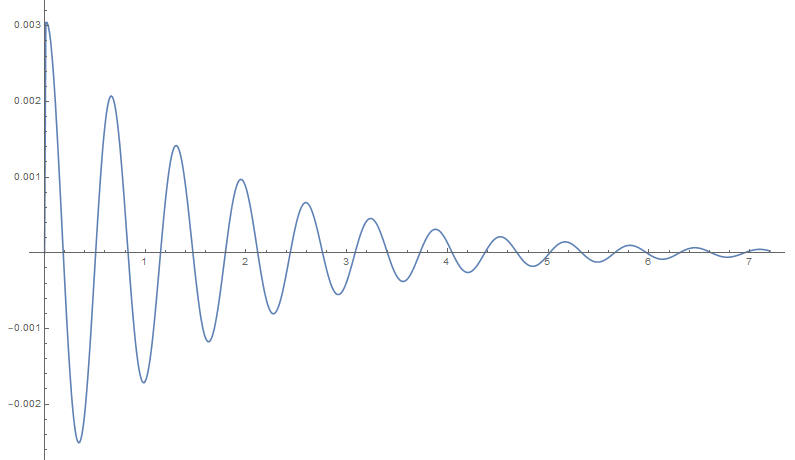}
		\caption{The evolution in time of $\varphi$ at $x=0.5$}\label{fig:4}

	\end{center}
\end{figure}
\begin{figure}
	\begin{center}
		\includegraphics[width=5in]{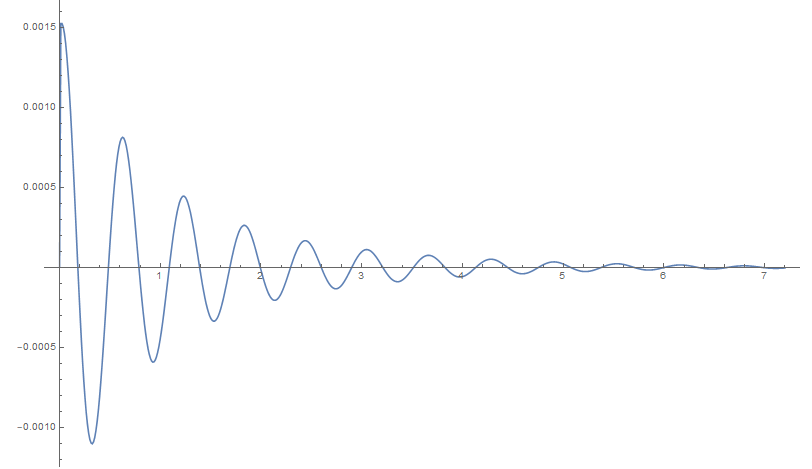}
		\caption{The evolution in time of $\psi$ at $x=0.5$}\label{fig:5}

	\end{center}
\end{figure}
\begin{figure}
	\begin{center}
		\includegraphics[width=5in]{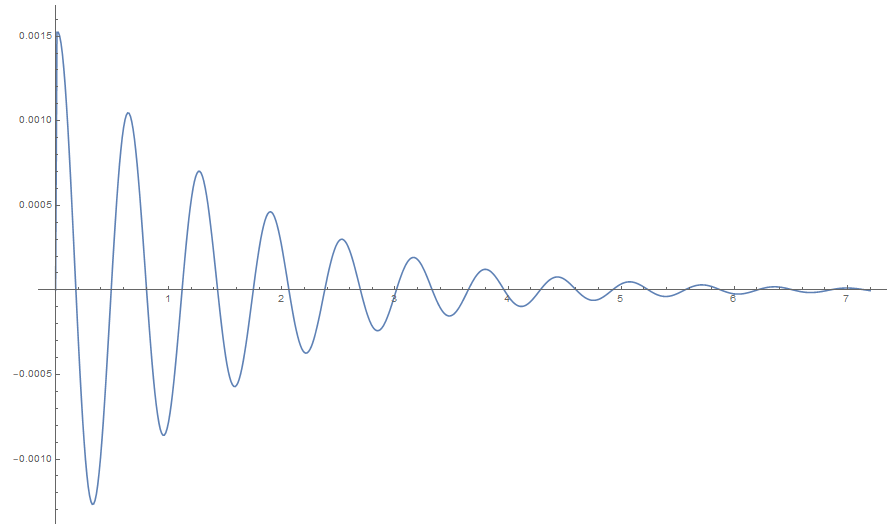}
		\caption{The evolution in time of $w$ at $x=0.5$}\label{fig:6}

	\end{center}
\end{figure}
%\paragraph{Implementation of the energy}

For the energy we have two cases, taking the conditions of equal and non-equl speeds of wave propagation. 

If $\frac{k _{1}}{k _{2}}=\frac{\rho _{1}}{\rho _{2}}$ and $k
_{1}=k _{3}$ we obtain an exponential decay by using the same data taken for the solutions as shown in the following Figures \ref{fig:7} -- \ref{fig:10}.
\begin{figure}
	\begin{center}
		\includegraphics[width=5in]{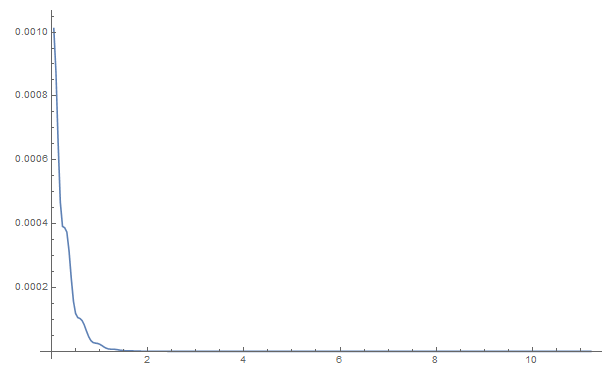}
		\caption{The evolution in time of $E^n$}\label{fig:7}

	\end{center}
\end{figure}
\begin{figure}
	\begin{center}
		\includegraphics[width=5in]{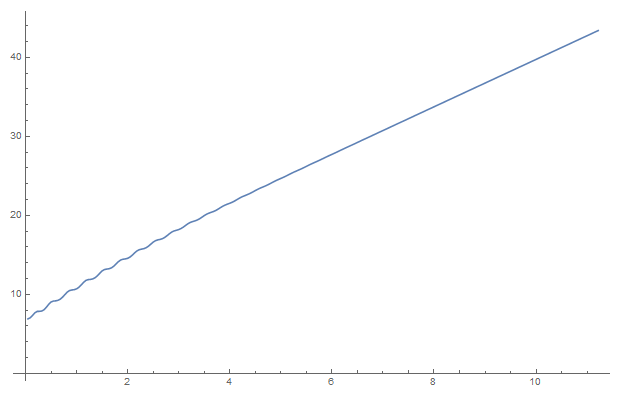}
		\caption{The evolution in time of $ln(E^n)$ that shows the exponential decay}\label{fig:8}

	\end{center}
\end{figure}
\begin{figure}
	\begin{center}
		\includegraphics[width=5in]{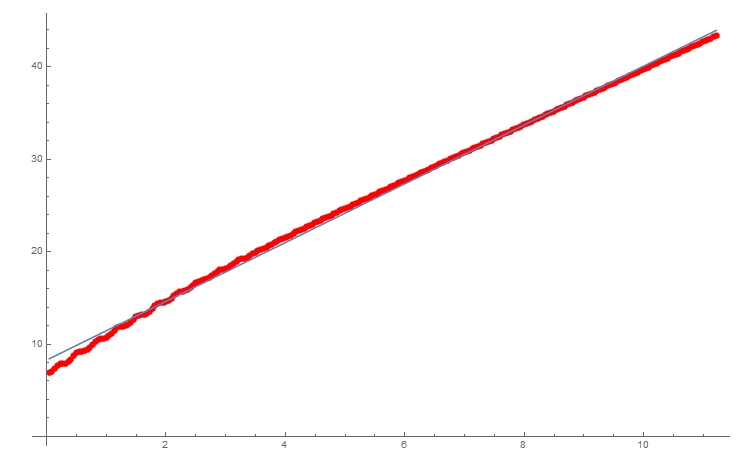}
		\caption{The evolution in time of $ln(E^n)$ with it's regression line}\label{fig:9}

	\end{center}
\end{figure}
\begin{figure}
	\begin{center}
		\includegraphics[width=5in]{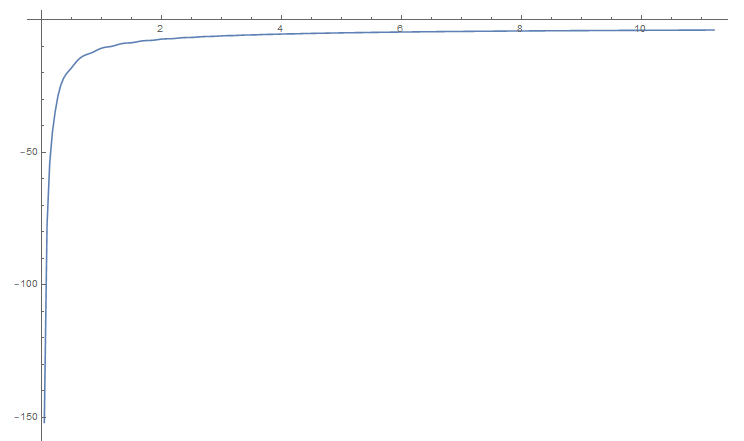}
		\caption{The evolution in time of $ln(E^n)/t$ }\label{fig:10}

	\end{center}
\end{figure}

If $\frac{k _{1}}{k _{2}}\neq \frac{\rho _{1}}{\rho _{2}}$ and $%
k _{1}\neq k _{3}$ we obtain a polynomial decay by taking the following data
$k _{1}=5$, $k _{2}=k _{3}=1,$ $\rho _{1}=0.02$, $\rho _{2}=0.1,$ $\Delta
t=0.03125,$ $h=0.0625,~$ and total time $T=16.4.$ with $g(x)=1/(x+1)^2$
 as show in the following Figures \ref{fig:11} -- \ref{fig:13}.
\begin{figure}
	\begin{center}
		\includegraphics[width=5in]{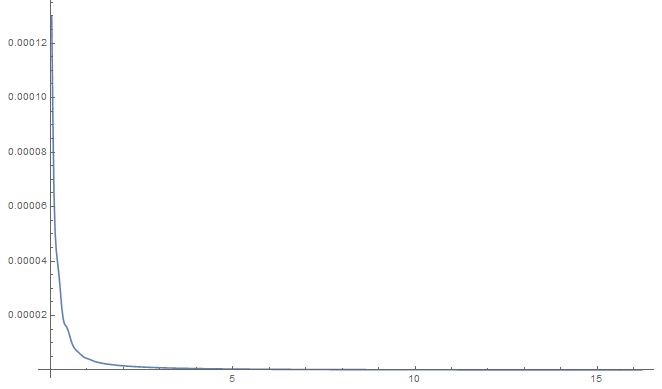}
		\caption{The evolution in time of $E^n$ }\label{fig:11}

	\end{center}
\end{figure}
\begin{figure}
	\begin{center}
		\includegraphics[width=5in]{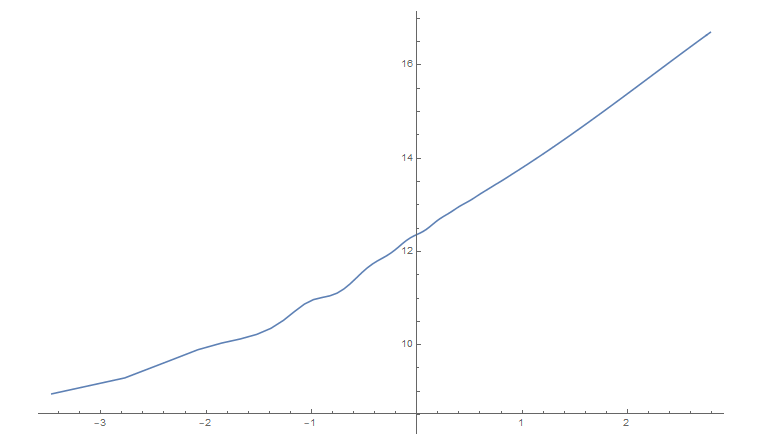}
		\caption{The variation of $-ln(E^n)$ with respect to $ln(t)$}\label{fig:12}

	\end{center}
\end{figure}
\begin{figure}
	\begin{center}
		\includegraphics[width=5in]{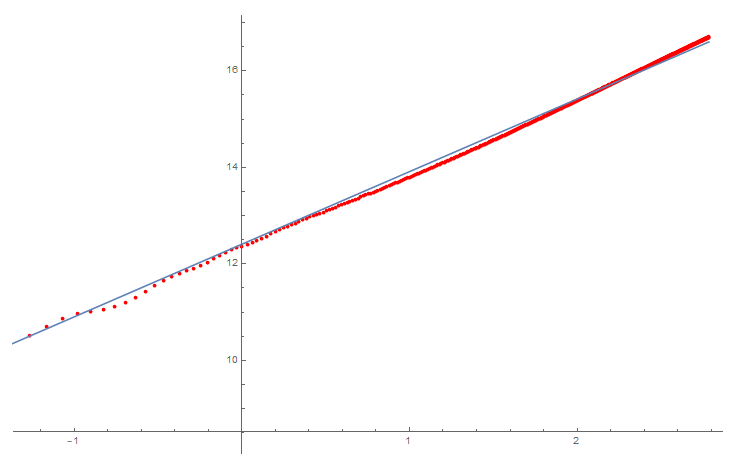}
		\caption{The variation of $-ln(E^n)$ with respect to $ln(t)$ with it's regression line}\label{fig:13}

	\end{center}
\end{figure}
\bigskip

%===========================================================================================
%+++++++++++++++++++++++++++++++++++++++++++++++++++++++++++++++++++++++++++++++++++++
%===========================================================================================

\section*{\small{Acknowledgement}}
%\textbf{Acknowledgements}\\
The authors would like to express their gratitude to King Fahd University of Petroleum and Minerals (KFUPM) for its continuous support. This work is partially funded by KFUPM under Project SB181018.

%===========================================================================================
%+++++++++++++++++++++++++++++++++++++++++++++++++++++++++++++++++++++++++++++++++++++
%===========================================================================================

\bibliographystyle{acm}
\bibliography{Finite_Memory_Bresse_System}
\end{document}